\documentclass[11pt, a4paper]{article}
\usepackage[margin=1in, top=1.5in, bottom=1.5in]{geometry}
\usepackage{subcaption}
\usepackage{graphicx}
\usepackage{float}
\usepackage{titlesec}
\usepackage{amsmath}
\usepackage{amssymb}
\usepackage{amsthm}
\usepackage{cases}
\usepackage[english]{babel}
\usepackage{hyperref}
\usepackage{enumitem}
\usepackage{authblk}
\usepackage{dsfont}
\usepackage{sectsty}
\usepackage{xcolor}

\usepackage{todonotes}

\theoremstyle{break}
	\newtheorem{thm}{Theorem}[section]
\theoremstyle{break}
    \newtheorem{prp}{Proposition}[section]
\theoremstyle{break}
	\newtheorem{asm}{Assumption}[section]
\theoremstyle{break}
	\newtheorem{rem}{Remark}[section]
\theoremstyle{break}
	\newtheorem{lem}{Lemma}[section]
\theoremstyle{break}
	
    \theoremstyle{break}
	\newtheorem{con}{Conclusion}[section]
\theoremstyle{break}
	
\theoremstyle{break}
	\newtheorem{expe}{Experiment}[section]
\theoremstyle{break}
    \newtheorem{main_contrib}{Contribution}

\makeatletter
\newcommand{\mylabel}[2]{#2\def\@currentlabel{#2}\label{#1}}
\makeatother

\counterwithin{equation}{section}

\title{Extinction and persistence criteria in non-local Klausmeier model of vegetation dynamics on flat landscapes}

\author[1,3]{Maciej Tadej \thanks{maciej.tadej@math.uni.wroc.pl}}
\author[1,2,4]{Ricardo Martinez-Garcia \thanks{r.martinez-garcia@hzdr.de}}
\author[1,3]{Michael Hecht \thanks{m.hecht@hzdr.de}}

\affil[1]{\textit{Center for Advanced Systems Understanding (CASUS) -- Helmholtz-Zentrum
    Dresden-Rossendorf (HZDR), Untermarkt 20, G\"orlitz 02826, Germany}}
\affil[2]{\textit{ICTP South American Institute for Fundamental Research \& Instituto de F\'isica Te\'orica, Universidade
Estadual Paulista - UNESP, São Paulo SP, Brazil}}
\affil[3]{\textit{Institute of Mathematics, University of Wrocław, pl. Grunwaldzki 2, 50-384 Wrocław, Poland}}
\affil[4]{\textit{Department of Ecology, Institute of Biosciences, University of S\~ao Paulo, S\~ao Paulo, R. do Mat\~ao, 321 - Butant\~a, 05508-090 , S\~ao Paulo, Brazil}}

\date{}

\begin{document}

\maketitle

\hspace{20pt}

\begin{abstract}
This paper investigates the dynamics of vegetation patterns in water-limited ecosystems using a generalized Klausmeier model that incorporates non-local plant dispersal within a finite habitat. We establish the well-posedness of the system and provide a rigorous analysis of the conditions required for vegetation survival. Our results identify a critical patch size governed by the trade-off between local growth and boundary losses; habitats smaller than this threshold lead to inevitable extinction. Furthermore, we derive a critical maximal biomass density below which the population collapses to a desert state, regardless of the domain size. We determine stability criteria for stationary solutions and describe the emergence of stable, non-trivial biomass distributions. Numerical experiments comparing sub-Gaussian and super-Gaussian kernels confirm that non-local dispersal mechanisms, particularly those with fat tails, enhance ecosystem resilience by allowing vegetation to persist in smaller, fragmented habitats than predicted by classical local diffusion models.
\end{abstract}


\newpage

\section{Introduction}

Water-limited ecosystems often exhibit a variety of regular vegetation patterns \cite{Rietkerk2008,Tarnita2024,Martinez-Garcia2022} resulting from the interplay between various water-vegetation feedbacks acting at different scales \cite{Borgogno2009,Martinez-Garcia2014,Meron2018}. Importantly, these patterns often form in response to increasingly harsh conditions and can therefore be important for ecosystem survival and persistence \cite{Meron2016}. Due to the long time scales at which vegetation patterns form and change, mathematical modeling has been an essential tool to better understand how patterns form and how they affect ecosystem dynamics \cite{Meron2016,Martinez-Garcia2023}. 

Different modeling efforts have, for example, suggested that pattern shapes could be used as early indicators of impending desertification processes, as water-stressed vegetation consistently forms hexagonally distributed spots in hyper-arid conditions \cite{Lefever1997,VonHardenberg2001,Meron2004,Rietkerk2002}. More recent results indicate that patterns could alternatively enhance ecosystem resilience by allowing ecosystems to avoid sudden collapses through spatial reshaping of the vegetation cover \cite{Rietkerk2021}. Mathematical modeling has thus made key contributions to our understanding of self-organization in water-limited ecosystems. However, most of these models are formulated in ideal conditions, assuming infinite domains, and thus neglecting how more realistic boundary conditions could change both spatial patterns and vegetation survival (but see, e.g., \cite{Clerc2021,PintoRamos2023,PintoRamos2025}).

In less applied scenarios, researchers have studied population dynamics in bounded domains, mainly using extensions of the Fisher–KPP equation \cite{Cantrell1999,Fagan2009,Colombo2018,Dornelas2024} known as KISS models in the mathematical biology literature \cite{Skellam1951,Kierstead1953}. KISS models show that population survival in confined domains---or habitats in ecological terms---requires that the domain is large enough so that reproduction within the domain balances population loss due to density fluxes through the domain edges \cite{Cantrell2004}. For domains smaller than this critical size, population extinction is the only stable solution. A key difference between these studies and vegetation-pattern models is that the latter often present non-local terms that account for long-range processes responsible for vegetation growth (e.g., resource intake via roots or sibling dispersal \cite{Martinez-Garcia2014,Eigentler2018}). Therefore, models for vegetation pattern formation provide an ideal scenario to extend KISS models to cases in which population dynamics are spatially non-local. 

We investigate the relationship between non-local population dynamics in finite habitats and population survival. To establish this connection, we use a generalized formulation of Klausmeier's model for vegetation pattern formation in flat landscapes that accounts for non-local plant dispersal \cite{Eigentler2018,Klausmeier1999}. We extend this formulation to consider the more realistic case in which vegetation can only grow in a finite region that is always surrounded by desert. 

Our approach is novel, but some progress has already been made. For instance, in \cite{JaramilloMeraz2024,CappaneraJaramilloWard2024} the authors constructed weak, local-in-time solutions. In a similar framework, though not involving the water density effects directly, the authors have shown several important properties such as large-time asymptotics, existence of non-trivial stable patterns, and a critical domain size, see \cite{Tadej2025,BaiLiWang2025,ReimerBonsallMaini2015,Sun2023} and references therein. When the local diffusion of water is replaced by a non-local operator it is known that there is a unique, global-in-time solution \cite{LaurencotWalker2023}.

\subsection{The non-local Klausmeier model}

The classic Klausmeier model for vegetation patterns describes the coupled dynamics of the vegetation biomass $v := v(x,t)$ and the water concentration $w := w(x,t)$ as a reaction-diffusion system. 
Here, we extend the dimensionless form \cite{Eigentler2018} to the finite habitat case:  
\begin{numcases}{}
    v_t = d_v\mathcal{L} v + v^2 w - B v & \hspace{2mm} $(x, t) \in \Omega \times [0, T],$ \label{eq:u} \\
    w_t = d_w\Delta w - v^2w - w + A & \hspace{2mm} $(x, t) \in \Omega \times [0, T],$ \label{eq:v} \\
    v = 0 & \hspace{2mm} $(x,t) \in \mathbb{R}^n \setminus \Omega  \times [0,T], $ \label{eq:u_bound} \\
    w = 0 & \hspace{2mm} $(x,t) \in \partial\Omega \times [0,T],$ \label{eq:v_bound} \\
    v(0) = v_0 & \hspace{2mm} $x \in \Omega ,$ \label{eq:Unit} \\
    w(0) = w_0 & \hspace{2mm} $x \in \Omega\,, $ \label{eq:v_init}
\end{numcases}
where $T > 0$ is the evolution time, $\Omega \subset \mathbb{R}^n$ an open and bounded domain, 
 $d_v > 0$ models plant dispersal, $d_w > 0$ the water diffusion, $A > 0$ the average rainfall, and $B > 0$ the mortality. 

In order to render the model more realistic we consider a non-local dispersal operator $\mathcal{L}$, given as 
\begin{equation}
    \label{eq:DispersalOperatorDefinition_1}
    \mathcal{L} v(x) = \mathds{1}_{\Omega}(x)\int_{\mathbb{R}^n} J(x-y)\left(v(y) - v(x)\right)\:dy\,.
\end{equation}
Here, $\mathds{1}_\Omega$ is the characteristic function of the domain $\Omega$ and $J : \mathbb{R}^n \rightarrow \mathbb{R}_{\geq 0}$ is an integral kernel fulfilling the assumptions listed below.

\begin{asm}
\label{assumptions_general_kernel}
We assume the integral kernel $J \in C^0(\mathbb{R}^n)$ satisfies the following properties:
\begin{enumerate}[label=\textnormal{(\arabic*)}, leftmargin=2em]
    \item[\mylabel{J1}{(J1)}] \textbf{Positivity:} 
    $J(z) \geq 0$ for all $z \in \mathbb{R}^n$, and $J(0) > 0$. 

    \item[\mylabel{J2}{(J2)}] \textbf{Symmetry:} 
    $J(z) = J(-z)$ for all $z \in \mathbb{R}^n$. 

    \item[\mylabel{J3}{(J3)}] \textbf{Decay:} 
    $J(z) \to 0$ as $|z| \to \infty$. 

    \item[\mylabel{J4}{(J4)}] \textbf{Finite second moment:} $\int_{\mathbb{R}^n} J(z)\,|z|^2 \,dz < \infty.$

    \item[\mylabel{J5}{(J5)}] \textbf{Normalization:} $\int_{\mathbb{R}^n} J(z)\,dz = 1.$
\end{enumerate}
\end{asm}

Assumption \ref{J5} and the non-local Dirichlet boundary condition \eqref{eq:u_bound} allow us to rewrite \eqref{eq:DispersalOperatorDefinition_1} as
\begin{equation}
    \label{EQ:DispersalOperatorDefinition_2}
    \mathcal{L} v(x) = \mathds{1}_{\Omega}(x) \left(\int_{\Omega} J(x-y) v(y)\:dy - v(x)\right).
\end{equation}

We furthermore impose the following constraints on the domain $\Omega$.

\begin{asm}
\label{assumptions_domain}
We assume that $\Omega \subset \mathbb{R}^n$ is bounded, open and satisfies the following properties:
\begin{enumerate}[label=\textnormal{(\arabic*)}, leftmargin=2em]
    \item[\mylabel{O1}{(O1)}] \textbf{Uniform estimates:} 
    Let $\{\phi_j\}_{j \in \mathbb{N}}$ be the eigenfunctions of $\Delta$ with Dirichlet boundary conditions. There exists a constant $C := C(\Omega) > 0$ such that $\sup_{j \in \mathbb{N}}\|\phi_j\|_\infty \leq C$.

    \item[\mylabel{O2}{(O2)}] \textbf{Boundary regularity:} 
    The boundary of $\Omega$ is sufficiently regular, i.e.\ $\partial\Omega \in C^k$ with suitable $k \in \mathbb{N}$.
\end{enumerate}
\end{asm}

For the convenience of the reader, we also provide a summary of the notation used throughout this article.

\begin{main_contrib}[Critical Patch Size and Kernel Shape]
We derive an analytical criterion for vegetation persistence based on the spectral properties of the dispersal operator. Specifically, as shown in Lemma \ref{lem:apriori_smallnes_d_v}, extinction is guaranteed if the dispersal rate $d_v$ exceeds a threshold $M/\beta_1$, where $M$ is related to the maximum growth rate and $\beta_1$ is the principal eigenvalue of the operator $-\mathcal{L}$ on $\Omega$. Since $\beta_1$ scales inversely with the domain size (see Remark \ref{rem:dispersal_patch_condition}), this establishes a \textit{critical patch size} for persistence. 

Crucially, our numerical analysis reveals that this critical size is sensitive to the shape of the dispersal kernel. We show that non-local kernels with heavier tails (i.e., higher kurtosis) significantly reduce the critical patch size compared to the classical Laplacian diffusion (see Figure \ref{fig:critical_patch_size}, left panel). This implies that long-range dispersal mechanisms enhance ecosystem resilience in fragmented landscapes.
\end{main_contrib}

\begin{main_contrib}[Minimum Biomass Threshold for Persistence]
We identify a critical threshold for the initial biomass density required to prevent ecosystem collapse. Our stability analysis, formalized in Theorem \ref{thm:Exinction-Of-Vegetation}, demonstrates that if the maximal biomass density drops below a critical level, approximately scaled by the ratio of mortality to rainfall, $\|V\|_\infty \sim B/A$, then the system is attracted to the trivial desert state regardless of the habitat size. This theoretical lower bound is numerically confirmed by the bifurcation diagrams in Figures \ref{fig:bifurcation_diagram_slow} and \ref{fig:bifurcation_diagram_fast}, which show that no stable non-trivial solutions exist below this biomass threshold.
\end{main_contrib}

\begin{main_contrib}[Existence and Stability of Non-Uniform Patterns]
We prove the existence of stable, non-trivial stationary solutions $V(x)$ that bifurcate from the spatially uniform states of the kinetic system.
\begin{itemize}
    \item \textbf{Existence Condition:} Non-deserted states emerge only when the environmental stress is low enough, specifically when rainfall $A$ and mortality $B$ satisfy $A > 2B$. Under this condition, and ensuring the reaction kinetics satisfy specific bounds (established in Proposition \ref{prp:nonlinearity_f_properties_at_v_star}), Theorem \ref{thm:Stationary-Solutions-Existence} guarantees the existence of a stationary solution $V$ in the neighborhood of the uniform state $v_*$.
    \item \textbf{Stability Regime:} By analyzing the linearized operator in Theorem \ref{thm:Stationary-Solutions-Stability}, we show that these spatially structured solutions are linearly exponentially stable, provided that the transport parameters $d_v$ and $d_w$ are sufficiently small.
\end{itemize}
Unlike the classical reaction-diffusion model, which predicts smooth profiles, the non-local model supports stable solutions characterized by sharp biomass gradients at the habitat boundaries (as visualized in the right panel of Figure \ref{fig:critical_patch_size}), offering a distinct morphological signature of non-local interactions.
\end{main_contrib}

\subsection{Structure}
The remainder of this article is organized as follows. 
In Section 2, we establish the well-posedness of the non-local Klausmeier model by proving the local-in-time existence and uniqueness of solutions to the problem (1.1)--(1.6). We also demonstrate fundamental properties such as the non-negativity of solutions and the existence of an invariant region. 
Section 3 contains the core mathematical analysis of this paper, where we provide a rigorous proof for the existence of non-trivial stationary solutions. 
In Section 4, we investigate the long-term behavior of these solutions by deriving criteria for both the extinction of vegetation and the persistence of these equilibria through a stability analysis. 
Section 5 is devoted to numerical experiments designed to validate and illustrate our theoretical findings. We present two main computational studies: the first investigates the critical patch size for different dispersal kernels, while the second explores the system's bifurcation structure. 
Finally, in Section 5.1, we summarize our main conclusions and outline potential directions for future research.

\section{Existence of solutions}
In order to show that the non-local Klausmeier model \eqref{eq:u}--\eqref{eq:v_init} is well-posed we start by constructing local-in-time solutions. 

\subsection{Local-in-time solutions} 
We denote with 
\begin{equation}
    \label{EQ:C_0_OMEGA}
    C_{0, \Omega}(\mathbb{R}^n) = \left\{ v \in L^\infty(\mathbb{R}^n): v \in C^0(\Omega), \: v = 0 \text{ in } \mathbb{R}^n \setminus \Omega \right\},
\end{equation}
the space of functions vanishing outside of $\Omega$, which becomes a Banach space when equipped with the norm $\|v\|= \|v\|_{C^0(\Omega)} + \|v\|_{L^\infty(\mathbb{R}^n\setminus \Omega)}$. We study the linear part in the vegetation equation \eqref{eq:u}.  

\begin{lem} 
Let $\mathcal{L}:  C_{0, \Omega}(\mathbb{R}^n) \to C_{0, \Omega}(\mathbb{R}^n)$ be the non-local dispersal operator from equation~\eqref{eq:DispersalOperatorDefinition_1}, $B>0$.
Consider the linear operator $\mathcal{A}: C_{0, \Omega}(\mathbb{R}^n) \to C_{0, \Omega}(\mathbb{R}^n)$ given as
\begin{align}
    \label{EQ:A_OPERATOR_DEFINITION}
    \mathcal{A} v (x) =
    \begin{cases}
        d_v\mathcal{L}v - Bv &x \in \Omega, \\
        0 &x \in \mathbb{R}^n \setminus \Omega,
    \end{cases} \quad d_v >0 \,.
\end{align} 
We denote with $\mathcal{T}(t) : C_{0, \Omega} (\mathbb{R}^n) \to  C_{0, \Omega} (\mathbb{R}^n)$, $t \geq 0$, the solution operator of the linear problem 
    \begin{equation}
        \label{eq:v_semilinear_problem}
        \frac{d}{dt} v(t) = \mathcal{A} v (t) \,, \quad t > 0 \,, \quad  v(0) = v_0 \in C_{0, \Omega}.
    \end{equation}
Then $\mathcal{T}(t) = e^{t\mathcal{A}}$ is well defined and the family $\{\mathcal{T}(t)\}_{t \geq 0}$ is a uniformly continuous semigroup of contractions.
    \label{lem:A_SEMIGROUP_GENERATOR}
\end{lem}

\begin{proof} 
    First, observe that by definition of $\mathcal{L}$, the operator $\mathcal{A}$ is bounded and linear. For each $v \in C_{0, \Omega}$ the solution of the problem \eqref{eq:v_semilinear_problem} exists for all $t \geq 0$, it is unique and given explicitly by $u(t) = \mathcal{T}(t)u(0)$, with $\mathcal{T}(t) = e^{t\mathcal{A}} = \sum_{n=1}^\infty \frac{\mathcal{A}^n t^n}{n!}$. Hence, $\mathcal{T}(t)$ is well defined and $\{\mathcal{T}(t)\}_{t \geq 0}$ is a uniformly continuous semigroup, see \cite[Proposition 1.2.2]{cazenave1998introduction} and \cite[Theorem 1.2]{Pazy}. Thanks to  \cite[Proposition 3.2]{ChasseigneChavesRossi2006} and \cite[Theorem 2.1]{Tadej2025}, the spectrum of the operator $\mathcal{L}$ is discrete and negative. Thus, for each $t \geq 0$, $\mathcal{T}(t):  C_{0, \Omega}(\mathbb{R}^n) \rightarrow C_{0, \Omega}(\mathbb{R}^n)$ is a contraction \cite[Proposition V.1.7]{EngelNagel2000}.    
\end{proof}

Next we study the water equation \eqref{eq:v}. Denote by $C_0(\Omega)$ the Banach space of functions continuous in $\overline{\Omega}$ and vanishing on $\partial\Omega$.

\begin{lem}
    \label{lem:B_SEMIGROUP_GENERATOR}
    Consider the linear operator $\mathcal{B}: D(\mathcal{B}) \to C_0(\Omega)$ given by 
    \begin{align}
        \label{EQ:B_OPERATOR_DEFINITION}
        \mathcal{B} w (x) = 
        \begin{cases}
            d_w \Delta w - w &x \in \Omega, \\
            0 &x \in \partial \Omega, 
        \end{cases} \quad d_w > 0 \,, 
    \end{align}
    where the domain $D(\mathcal{B})$ is given by
    \begin{equation}
        \label{EQ:B_DOMAIN_DEFINITION}
        D(\mathcal{B}) = \left\{ w \in C^2_0(\Omega): \Delta w \in C_0(\Omega) \right\}\,.
    \end{equation}
    Then the solution operator $\mathcal{T}(t): D(\mathcal{B}) \to C_0(\Omega)$ given by $\mathcal{T}(t) = e^{t\mathcal{B}}$ is well defined and the family $\{\mathcal{T}(t)\}_{t \geq 0}$ is a strongly continuous semigroup of contractions.
\end{lem}
\begin{proof}
    The proof can be found in  \cite[Theorem 6.1.8]{Arendt2001}.
\end{proof}
We use the ingredients above to construct the local-in-time solutions of the Klausmeier system.
\begin{thm}
    \label{thm:Well-Posedness}
    Let $(v_0, w_0) \in C_{0, \Omega}(\mathbb{R}^n) \times C_0(\Omega)$. Then there exists $0 < T := T(v_0, w_0)$ and a unique solution $(v, w)$ to the problem \eqref{eq:u}--\eqref{eq:v_init}, such that
    \begin{equation}
        v \in C\left([0, T], C_{0, \Omega}(\mathbb{R}^n)\right) \cap C^1\left([0, T), C_{0, \Omega}(\mathbb{R}^n)\right),
    \end{equation}
    and
    \begin{equation}
        w \in C\left([0, T], C_0(\Omega)\right) \cap C^1\left([0, T), C_0(\Omega)\right).
    \end{equation}
\end{thm}
\begin{proof}
    Given $\mathcal{A}, \mathcal{B}$ by the formulas \eqref{EQ:A_OPERATOR_DEFINITION} and \eqref{EQ:B_OPERATOR_DEFINITION}--\eqref{EQ:B_DOMAIN_DEFINITION} respectively, let us define the operator $\mathcal{C}: C_{0, \Omega}(\mathbb{R}^n) \times D(\mathcal{B}) \to C_{0, \Omega}(\mathbb{R}^n) \times C_0(\Omega)$ as follows
    \begin{equation}
        \mathcal{C} = 
        \begin{bmatrix}
            \mathcal{A} & 0\\ 0 & \mathcal{B}
        \end{bmatrix}.
    \end{equation}
    Next, denote the vectors $u = (v, w)$ and $u_0 = (v_0, w_0)$. Using this notation we rewrite the Klausmeier system \eqref{eq:u}--\eqref{eq:v_init} as
    \begin{equation}
        \label{eq:ClassicalSemigroupFormulation}
         u'(t) = \mathcal{C} u(t) + F(u(t)), \quad u(0) = u_0,
    \end{equation}
    where  $F: C_{0, \Omega}(\mathbb{R}^n) \times C_0(\Omega) \to C_{0, \Omega}(\mathbb{R}^n) \times C_0(\Omega)$ is given by
    \begin{align}
        F(u) = F(v, w) = 
        \begin{cases}
            (v^2 w, -v^2 w + A) &x \in \Omega, \\
            (0, 0) &x \notin \Omega
        \end{cases}.
    \end{align}
    Lemma \ref{lem:A_SEMIGROUP_GENERATOR} and Lemma \ref{lem:B_SEMIGROUP_GENERATOR} yield that $\mathcal{C}$ generates a semigroup $\{\mathcal{T}(t)\}_{t \geq 0}$ of contractions on $C_{0, \Omega}(\mathbb{R}^n) \times C_0(\Omega)$. Thus, we construct the solution $u$ using the Duhamel formula \cite[Lemma 4.1.1]{cazenave1998introduction}
    \begin{equation}
        \label{eq:MildSemigroupFormulation}
         u(t) = \mathcal{T}(t) u_0 + \int_0^t \mathcal{T}(t-s) F(u(s)) \:ds.
    \end{equation}
    Since $u_0 \in C_{0, \Omega}(\mathbb{R}^n) \times C_0(\Omega)$ and $F$ is a locally Lipschitz function, we proceed with a standard Banach fixed point argument \cite[Proposition 4.3.3]{cazenave1998introduction}, ensuring that there exists a positive number $T := T(u_0)$ such that \eqref{eq:MildSemigroupFormulation} admits a unique solution $u \in C\left([0, T], C_{0, \Omega}(\mathbb{R}^n) \times C_0(\Omega)\right)$. To show differentiability in time one needs to integrate equations \eqref{eq:u}, \eqref{eq:v} to find that $u$ satisfies
    \begin{equation}
        \label{eq:IntegralFormula}
        u(t) = u_0 + \int_0^t (\mathcal{C}u(s) + F(u(s))) \:ds.
    \end{equation}
    Hence, the continuity of $u$ implies differentiability of $u$ in the time interval $[0, T)$.
\end{proof}
Given that solutions exist we continue by studying their properties. 

\subsection{Non-negative and bounded solutions} 
We start by showing the non-negativity and the existence of an invariant region $\Gamma$ for the system \eqref{eq:u}--\eqref{eq:v_init}.
\begin{thm}
    \label{thm:Non-negativity}
     Let the initial datum $(v_0, w_0) \in C_{0, \Omega}(\mathbb{R}^n) \times C_0(\overline{\Omega})$ be non-negative. Then the solution $(v, w)$ ensured by Theorem \ref{thm:Well-Posedness} is non-negative and for all $(x, t) \in \Omega \times [0, T]$ it holds
     \begin{equation}
         w(x,t) \leq \max\{\|w_0\|_\infty, A\} =: R_1 >0\,.
     \end{equation}
     Moreover, let $\Gamma = [0, B/R_1]$. If  $v_0(x) \in \Gamma$ for all $x \in \Omega$ then for all $(x, t) \in \Omega \times [0, T]$ the solution $v(x,t) \in \Gamma$.
\end{thm}
\begin{proof}
    It follows from the parabolic maximum principle that for all $(x, t) \in \Omega \times [0, T)$ we have the estimate of the form
    \begin{equation}
        \label{EQ:AprioriEstimate_V}
        0 \leq w(x, t) \leq \max\{\|w_0\|_\infty, A\} =: R_1.
    \end{equation}
    Next, we define the point $\xi := \xi(t) \in \Omega$ and evaluation $\mu := \mu(t)$ such that 
    \begin{equation*}
        \mu(t) := v(\xi(t), t) = \min_{x \in \Omega} v(x, t).
    \end{equation*}
    Without loss of generality, we take arbitrary $0 < s < t < T$ and use the 
    Lipschitz continuity of $\mu$ to estimate
    \begin{equation*}
        \left|\mu(t) - \mu(s)\right| \leq \left|v(\xi(s), t) - v(\xi(s), s)\right| \leq \max_{(x, r) \in \Omega \times [0, T]} |v(x, r)|\, |t-s|.
    \end{equation*}
    It follows from the Radamacher Theorem \cite[Theorem 3.2]{EvansGariepy2015} that $\mu$ is differentiable almost everywhere in $[0, T]$. Next, we estimate the difference quotients by applying the mean value theorem
    \begin{align*}
        \frac{\mu(t+h) - \mu(t)}{h} &\leq \frac{v(\xi(t+h), t+h) - v(\xi(t + h), t)}{h} = v_t(\xi(t + h), c(t, h)), \\
        \frac{\mu(t+h) - \mu(t)}{h} &\geq \frac{v(\xi(t), t+h) - v(\xi(t), t)}{h} = v_t(\xi(t), c(t, h)),
    \end{align*}
    where $c(t,h) \in [t, t+h]$ is an intermediate point. Passing with $h \rightarrow 0$, estimating $\mathcal{L}$ in a minimum point and using non-negativity of $w$, we obtain the differential inequality
    \begin{align}
        \label{eq:differential_inequality_min}
        \begin{split}
            \mu'(t) &= d_v \mathcal{L} v(\xi(t), t) + \mu^2(t) w(\xi(t), t) - B \mu(t) \geq - B \mu(t), \\
            \mu(0) &= \min_{x \in \Omega} v_0(x) \geq 0.
        \end{split}
    \end{align}
    When solved, inequality \eqref{eq:differential_inequality_min} yields the following lower bound 
    \begin{equation*}
        v(x,t) \geq \mu(t) \geq \mu(0) e^{-B t} \geq 0.
    \end{equation*}
    Now, we define the point $\zeta := \zeta(t) \in \Omega$ and evaluation $\nu := \nu(t)$ such that
    \begin{equation*}
        \nu(t) := v(\zeta(t), t) = \max_{x \in \Omega} v(x,t).
    \end{equation*}
    Proceeding as previously, we arrive at the following differential inequality
    \begin{align}
        \label{eq:differential_inequality_max}
        \begin{split}
            \nu'(t) &= d_v\mathcal{L} v(\zeta(t), t) + \nu^2(t) w(\zeta(t), t) - B \nu(t) \leq R_1 \nu^2(t) - B \nu(t), \\
            \nu(0) &= \max_{x \in \Omega} v_0(x) \leq B/R_1.
        \end{split}
    \end{align}
    When solved, inequality \eqref{eq:differential_inequality_max} yields the following upper bound
    \begin{equation}
        \label{EQ:AprioriEsitmate_Unvariant_Region}
        v(x,t) \leq \nu(t) \leq \frac{B\nu(0)}{R_1 \nu(0) + 
        (B - R_1 \nu(0)) e^{B t}} \leq B / R_1.
    \end{equation}
    Hence, $\Gamma = [0, B/R_1]$ is an invariant region.
\end{proof}
While so far we have focused on local-in-time solutions, we next study the long-time behavior of the Klausmeier system. In particular, we use the results above for the stability analysis of stationary solutions.  

\section{Stationary solutions}
We seek stationary solutions of the Klausmeier system \eqref{eq:u}--\eqref{eq:v_init}, i.e.\ solutions of the system  
\begin{numcases}{}
    0 = d_v \mathcal{L} v + v^2 w - B v & \hspace{2mm} $x \in \Omega,$ \label{eq:u_stationary} \\
    0 = d_w\Delta w - v^2w - w + A & \hspace{2mm} $x \in \Omega,$ \label{eq:v_stationary} \\
    v = 0 & \hspace{2mm} $x \in \mathbb{R}^n \setminus \Omega , $ \label{eq:u_bound_stationary} \\
    w = 0 & \hspace{2mm} $x \in \partial\Omega.$ \label{eq:v_bound_stationary}
\end{numcases}

We start by showing that there exist non-trivial stationary solutions with non-constant water profiles. 

\subsection{Existence of non-constant water profiles}

As a first step, we aim to reduce the number of equations in the Klausmeier system \eqref{eq:u_stationary}--\eqref{eq:v_bound_stationary}. 
 
\begin{lem}
    \label{lem:existence_of_inhomogeneous_elliptic_v}
    Let $v \in C(\Omega)$ be a given function. Then there exists a unique solution $W \in W^{1,2}_0(\Omega) \cap W^{2,2}_{loc}(\Omega)$ of equations \eqref{eq:v_stationary} and \eqref{eq:v_bound_stationary}. Moreover, we have that $\|W \|_\infty \leq A$, where $A >0$ is the average rainfall. 
\end{lem}
\begin{proof}
    The proof of the existence and uniqueness is standard, see for instance \cite[Theorem 8.9]{gilbarg2001elliptic}. The remaining estimate $\|W\|_\infty \leq A$ follows directly from the elliptic maximum principle, we refer to \cite[Chapter 3]{gilbarg2001elliptic}.
\end{proof}

Next we show that the mapping $W = W(v)$ is Lipschitz continuous.
\begin{lem}
    \label{lem:stationary_v_lipschitz_constant}
    Let $B_R \subset C(\Omega)$ be the closed ball of radius $R > 0$ centered at $0$. 
    Consider the map $W : B_R \rightarrow W^{1,2}_0(\Omega)$, with $W=W(v)$ being the solution of the problem \eqref{eq:v_stationary}, \eqref{eq:v_bound_stationary}. Then $W (B_R) \subset L^\infty(\mathbb{R}^n)$ and for arbitrary $v_1, v_2 \in B_R$ there holds
    \begin{equation}
        \|W(v_1) - W(v_2)\|_\infty \leq 2 A R \|v_1 - v_2\|_\infty.
    \end{equation}
\end{lem}

\begin{proof}
    For arbitrary $v_1, v_2 \in B_R \subset L^\infty(\Omega)$ we define the difference $u = W(v_1) - W(v_2)$, satisfying 
    \begin{align}
        \label{eq:Difference_ELLIPTIC_DEF}
        \begin{cases}
            d_w\Delta u - (v_1^2 + 1) W(v_1) + (v_2^2 + 1) W(v_2) = 0, &\quad x \in \Omega, \\
            u = 0, &\quad x \in \partial\Omega.
        \end{cases}
    \end{align}
    Since $W(v_1) = u + W(v_2)$, we can equivalently express problem \eqref{eq:Difference_ELLIPTIC_DEF} as
    \begin{align}
        \label{eq:lipschitz_continuity_formulation_difference}
        \begin{cases}
            d_w\Delta u - (v_1^2 + 1) u  = W(v_2) (v_1^2 - v_2^2), &\quad x \in \Omega, \\
            u = 0, &\quad x \in \partial\Omega.
        \end{cases}
    \end{align}
    We estimate the supremum norm of the right-hand side of the first equation in \eqref{eq:lipschitz_continuity_formulation_difference}
    \begin{equation}
        \label{eq:Linfty_norm_rhs_difference_estimate}
        \|W(v_2) (v_1^2-v_2^2)\|_\infty \leq 2AR\|v_1-v_2\|_\infty.
    \end{equation}
    Note that by Lemma \ref{lem:existence_of_inhomogeneous_elliptic_v} $u \in W^{1,2}(\Omega)$. Thus, by regularity theory \cite[Theorems 8.33 and 8.34]{gilbarg2001elliptic} and the bound \eqref{eq:Linfty_norm_rhs_difference_estimate} we obtain that the solution $u$ of the problem \eqref{eq:lipschitz_continuity_formulation_difference} satisfies
    \begin{equation}
        \|W(v_1) - W(v_2)\|_{C^{1, \alpha}(\Omega)} =: \|u\|_{C^{1, \alpha}(\Omega)}  \leq 2AR\|v_1-v_2\|_\infty,
    \end{equation}
    with $0 < \alpha < 1$. Finally, the definition of the norm $\|\cdot\|_{C^{1,\alpha}(\Omega)}$ implies that 
    \begin{equation}
        \|W(v_1) - W(v_2)\|_\infty = \|u\|_\infty \leq \|u\|_{C^{1, \alpha}(\Omega)}  \leq 2AR\|v_1-v_2\|_\infty.
    \end{equation}
\end{proof}
In light of the fact above, we consider the following system of constant solutions of equations \eqref{eq:u_stationary} and \eqref{eq:v_stationary}
\begin{align}
    \label{EQ:ConstantSteadyStates}
    \begin{cases}
        0 = v^2w - Bv,\\
        0 = -v^2w - w + A.
    \end{cases}
\end{align}
We show that the solution $W(v_*)$, obtained by Lemma \ref{lem:existence_of_inhomogeneous_elliptic_v}, is close to the corresponding solution $(v_*, w_*)$ of the system \eqref{EQ:ConstantSteadyStates}. 

\begin{lem}
    \label{lem:stationary_v_estimate_v_star}
    Let $(v_*, w_*)$ be a real solution of the system \eqref{EQ:ConstantSteadyStates}. Denote by $W(v_*)$ the solution of equations \eqref{eq:v_stationary} and \eqref{eq:v_bound_stationary}. Then for all $\epsilon > 0$ there exists a diffusion rate $d_w > 0$, sufficiently small, such that
    \begin{equation*}
        \| W(v_*) - w_*\|_\infty < \epsilon.
    \end{equation*}
\end{lem}

\begin{proof}
    We denote with $\{\phi_j\}_{j \in \mathbb{N}} \subset L^2(\Omega)$ the set of eigenfunctions of $\Delta$ with Dirichlet boundary condition. Next, 
    we expand a constant function $A = \sum_{j=1}^\infty a_j \phi_j$. Since $W(v_*) \in L^2(\Omega)$ we can expand it in the eigenfunction basis as well, namely $W(v_*) = \sum_{j=1}^\infty c_j \phi_j$. Plugging this into equation \eqref{eq:v_stationary} we get that
    \begin{equation}
        \label{eq:Eigenfunctions_expansion}
        d_w\sum_{j=1}^\infty c_j \lambda_j \phi_j - (v_*^2 + 1) \sum_{j=1}^\infty c_j \phi_j + \sum_{j=1}^\infty a_j \phi_j = 0.
    \end{equation}
    We multiply equation \eqref{eq:Eigenfunctions_expansion} by each $\phi_j$ and then integrate. From orthogonality, we get that the following equation holds for all $j \in \mathbb{N}$
    \begin{equation}
        \label{eq:Eigenfunctions_expansion_coeffs}
        -d_w c_j\lambda_j - (v_*^2+1) c_j + a_j = 0.
    \end{equation}
    Now, solving equation \eqref{eq:Eigenfunctions_expansion_coeffs} for the coefficients $c_j$ we obtain
    \begin{equation*}
        c_j = \frac{a_j}{v_*^2 + 1 + d_w\lambda_j}.
    \end{equation*}
    It follows that
    \begin{equation*}
        W(v_*) = \sum_{j=1}^\infty \frac{a_j}{v_*^2 + 1 + d_w\lambda_j} \phi_j.
    \end{equation*}
    Next, we plug the eigenfunction expansion of the constant function $A$ into the definition of the steady state $w_* = \frac{A}{v_*^2 + 1}$. Observe that we assumed that $\Omega$ is such that $\{\phi_j\}_{j \in \mathbb{N}}$ is uniformly bounded by some number $C > 0$. This allows us to estimate
\begin{align*}
    \|W(v_*) - w_*\|_\infty 
    &= 
    \left\|
        \sum_{j=1}^\infty 
        \frac{a_j}{v_*^2 + 1 + d_w \lambda_j} \phi_j 
        - 
        \sum_{j=1}^\infty 
        \frac{a_j}{v_*^2 + 1} \phi_j
    \right\|_\infty 
    \\[6pt]
    &=
    \left\|
        \sum_{j=1}^\infty
        \left(
            \frac{1}{v_*^2 + 1} 
            - 
            \frac{1}{v_*^2 + 1 + d_w \lambda_j}
        \right) 
        a_j \phi_j
    \right\|_\infty
    \\[6pt]
    &\leq
    C 
    \sum_{j=1}^\infty 
    \left| 
        \frac{d_w \lambda_j a_j}{(v_*^2 + 1)(v_*^2 + 1 + d_w \lambda_j)}
    \right|
    \\[6pt]
    &=
    C 
    \sum_{j=1}^N 
    \left| 
        \frac{d_w \lambda_j}{v_*^2 + 1 + d_w \lambda_j}
    \right|
    \left| 
        \frac{a_j}{v_*^2 + 1} 
    \right|
+
    C
    \sum_{j=N+1}^\infty 
    \left| 
        \frac{d_w \lambda_j}{v_*^2 + 1 + d_w \lambda_j}
    \right|
    \left| 
        \frac{a_j}{v_*^2 + 1} 
    \right|.
\end{align*}

    For any fixed $\epsilon > 0$ there exists $N \in \mathbb{N}$ such that $\sum_{j=N+1}^\infty |a_j| < \frac{\epsilon}{2C}$. Hence, provided that $d_w < \frac{\epsilon}{2C} \left(\sum_{j=1}^N \frac{|\lambda_j a_j|}{(v_*^2 + 1)^2}\right)^{-1}$ we obtain the estimate
    \begin{equation}
        \|W(v_*) - w_* \|_\infty < \epsilon.
    \end{equation}
\end{proof}

Having that, we continue our analysis of stationary solutions by analyzing the existence of stationary vegetation solutions.

\subsection{Existence of non-constant vegetation profiles}

We use Lemmas \ref{lem:existence_of_inhomogeneous_elliptic_v} and \ref{lem:stationary_v_lipschitz_constant} to reduce the system \eqref{eq:u_stationary}--\eqref{eq:v_bound_stationary} to the following non-local and non-linear problem
\begin{numcases}{}
    0 = d_v \mathcal{L}v + f(v) & \hspace{2mm} $x \in \Omega,$ \label{eq:v_stationary_redef} \\
    v = 0 & \hspace{2mm} $x \in \partial\Omega,$ \label{eq:v_bound_stationary_redef}
\end{numcases}
where we denoted $f(v) = v^2W(v) - Bv$. Note that $W(v)$ is the stationary solution obtained by Lemma \ref{lem:existence_of_inhomogeneous_elliptic_v}.
\begin{lem}
    \label{lem:apriori_smallnes_d_v}
    Suppose that $f(0) = 0$ and $f$ is Lipschitz continuous with constant $M$. Denote by $\beta_1$ the principal eigenvalue of $-\mathcal{L}$. If
    \begin{equation}
        \label{eq:d_v_condition_1}
        d_v > \frac{M}{\beta_1},
    \end{equation}
    then there exists only one stationary solution $v \in L^2(\Omega)$ of the system \eqref{eq:v_stationary_redef}--\eqref{eq:v_bound_stationary_redef}, namely  $v \equiv 0$.
\end{lem}
\begin{proof}
    Let $v \in L^2(\Omega)$ be a non-zero solution of \eqref{eq:v_stationary_redef}. We multiply equation \eqref{eq:v_stationary_redef} with $v$ and integrate. Applying the mean value theorem we estimate 
    \begin{equation*}
        0
        =
        d_v \langle \mathcal{L} v, v \rangle + \langle f(v), v \rangle
        =
        d_v \frac{\langle \mathcal{L} v, v \rangle}{\|v\|_2^2}\|v\|_2^2 + \langle f(v) - f(0), v \rangle
        \leq
        (-d_v \beta_1 + M)\|v\|_2^2.
    \end{equation*}
    Since $-d_v \beta_1 + M < 0$ it follows that $\|v\|_2 = 0$.
\end{proof}

\begin{rem}
    \label{rem:dispersal_patch_condition}
    From Lemma~\ref{lem:apriori_smallnes_d_v}, we can draw the following conclusions:
    \begin{itemize}
        \item For a fixed patch $\Omega \subset \mathbb{R}^n$, if the dispersal rate $d_v > 0$ is sufficiently large, the vegetation will go extinct. 
        \item For a fixed dispersal rate $d_v > 0$ the principal eigenvalues of $\mathcal{L}$ on any patches 
        $\Omega_1 \subset \Omega_2 \subset \mathbb{R}^n$ satisfy 
        $\beta_1(\Omega_2) < \beta_1(\Omega_1)$. 
        Hence, there exists a sufficiently large patch $\Omega$ such that vegetation will persist. 
    \end{itemize}
\end{rem}

\begin{thm}
    \label{thm:Stationary-Solutions-Existence}
    Denote by $(v_*, w_*)$ a non-zero solution of the system \eqref{EQ:ConstantSteadyStates}. Let $f : C_{0, \Omega}(\mathbb{R}^n) \rightarrow C_{0, \Omega}(\mathbb{R}^n)$ be twice Fr\'echet differentiable mapping. Assume that there exist sufficiently small $\epsilon > 0$ and numbers $M > 0, m > 0$ such that
    \begin{equation}
        \label{eq:nonlinearity_f_assumptions}
        \|f(v_*)\|_\infty \leq \epsilon, 
        \quad
        f'(v_*) < -m
        \quad
        \text{ for all }
        x \in \Omega
        \quad
        \text{ and }
        \quad
        \|f''\|_\infty < M.
    \end{equation}
    Then there exist $R_* > 0$ and $D_v > 0$ such that for all $d_v < D_v$ there exists a unique, non-zero solution $V$ of the problem \eqref{eq:v_stationary_redef}--\eqref{eq:v_bound_stationary_redef}, satisfying $\|V-v_*\|_\infty \leq R_*$.
\end{thm}

\begin{proof}
    First, let us denote the closed ball $B_R = \{ \varphi \in C_{0, \Omega}(\mathbb{R}^n) : \|\varphi\|_\infty \leq R \}$. 
    Let $V$ be a function defined by
    \begin{equation}
        \label{eq:stationary_u_definition}
        V(x) = v_* \mathds{1}_{\Omega}(x) + \varphi(x),
    \end{equation}
    where $\varphi \in B_R$. Using the Taylor expansion for the nonlinearity $f$, we can write
    \begin{equation}
        \label{eq:f_taylor_expansion}
        f(V) = f(v_*) + f'(v_*) \varphi + \mathcal{N}(\varphi),
    \end{equation}
    where $\mathcal{N}(\varphi)$ represents the higher-order remainder term. Assuming the second derivative is bounded by $M$, we have the estimate $\|\mathcal{N}(\varphi)\|_\infty \leq \frac{M}{2} \|\varphi\|_\infty^2$. For simplicity, let us denote $C_2 = \frac{M}{2}$.
    
    Since $f'(v_*) < -m < 0$ and $\mathcal{L}$ is negative definite, the linear operator $-d_v \mathcal{L} - f'(v_*) \mathrm{Id}$ is invertible, with the norm of the inverse bounded by $(m - \beta_1 d_v)^{-1}$. Plugging \eqref{eq:f_taylor_expansion} into the system \eqref{eq:v_stationary_redef}--\eqref{eq:v_bound_stationary_redef} allows us to reformulate the problem as a fixed point equation
    \begin{equation}
        \label{eq:fixed_point_formulation}
        \varphi = \left( -d_v \mathcal{L} - f'(v_*) \mathrm{Id} \right)^{-1} \left( d_v v_* \mathcal{L} \mathds{1}_\Omega + f(v_*) + \mathcal{N}(\varphi) \right) =: \mathcal{F}(\varphi).
    \end{equation}
    We first show that $\mathcal{F}: B_R \to B_R$. Note that we assumed \eqref{eq:nonlinearity_f_assumptions}, hence estimating the supremum norm of \eqref{eq:fixed_point_formulation}, we obtain
    \begin{align}
        \label{eq:f_ball_to_ball_condition_1}
        \| \mathcal{F}(\varphi) \|_\infty 
        &\leq
        \frac{d_v v_* \|\mathcal{L}\mathds{1}_\Omega\|_\infty + \epsilon + C_2 R^2}{m - \beta_1 d_v}.
    \end{align}
    Furthermore, we demand $\|\mathcal{F}(\varphi)\|_\infty \leq R$. The minimal radius $R_*$ satisfying this quadratic inequality corresponds to the smaller root:
    \begin{equation}
        \label{eq:f_ball_to_ball_radius}
        R_* = \frac{m - \beta_1 d_v}{2 C_2}
        \left(1 - 
        \sqrt{ 
        1 
        -
        \frac{4 C_2 (d_v v_* \|\mathcal{L}\mathds{1}_\Omega\|_\infty + \epsilon)}{(m - \beta_1 d_v)^2}
        }
        \right).
    \end{equation}
    From Lemma \ref{lem:stationary_v_estimate_v_star}, the term under the square root is positive and close to 1 for sufficiently small $d_v, d_w$, ensuring $R_*$ is well-defined and positive. Furthermore, $R_* \to 0$ as the parameters vanish, ensuring $R_* < v_*$ so that $V$ remains non-negative.

    Secondly, we verify that $\mathcal{F}$ is a contraction. Let $\varphi, \psi \in B_{R_*}$. We start by estimating the difference of the images:
    \begin{equation}
        \| \mathcal{F}(\varphi) - \mathcal{F}(\psi) \|_\infty 
        \leq \frac{1}{m - \beta_1 d_v} \| \mathcal{N}(\varphi) - \mathcal{N}(\psi) \|_\infty.
    \end{equation}
    To bound the difference of the nonlinear remainders, we re-write equation \eqref{eq:f_taylor_expansion} and use the Fundamental Theorem of Calculus:
    \begin{align}
        \begin{split}
            \mathcal{N}(\varphi) - \mathcal{N}(\psi) 
            &= \left( f(v_* + \varphi) - f(v_* + \psi) \right)  - f'(v_*) (\varphi - \psi)\\
            &=
            \int_0^1 \left( f'(v_* + \psi + t(\varphi - \psi)) - f'(v_*) \right) (\varphi - \psi) \, dt.
        \end{split}
    \end{align}
    Using the assumption that $\|f''\|_\infty \leq M$, we can bound the term in the integrand as $\|f'(v_* + \eta) - f'(v_*)\|_\infty \leq M \|\eta\|_\infty$. Applying this along with the triangle inequality for the convex combination $\|\psi + t(\varphi - \psi)\|_\infty \leq (1-t)\|\psi\|_\infty + t\|\varphi\|_\infty$, we obtain:
    \begin{align}
        \| \mathcal{N}(\varphi) - \mathcal{N}(\psi) \|_\infty 
        &\leq M \|\varphi - \psi\|_\infty \int_0^1 \left( (1-t)\|\psi\|_\infty + t\|\varphi\|_\infty \right) dt \\
        &= \frac{M}{2} (\|\varphi\|_\infty + \|\psi\|_\infty) \|\varphi - \psi\|_\infty.
    \end{align}
    Recalling that $C_2 = M/2$ and $\varphi, \psi \in B_{R_*}$, we arrive at the estimate:
    \begin{equation}
        \| \mathcal{N}(\varphi) - \mathcal{N}(\psi) \|_\infty \leq 2 C_2 R_* \|\varphi - \psi\|_\infty.
    \end{equation}
    Thus, the Lipschitz constant $k$ for the operator $\mathcal{F}$ is given by:
    \begin{equation}
        \label{eq:f_contraction_condition_1}
        k = \frac{2 C_2 R_*}{m - \beta_1 d_v}.
    \end{equation}
    Substituting the explicit formula for $R_*$ from \eqref{eq:f_ball_to_ball_radius} into \eqref{eq:f_contraction_condition_1}, we find:
    \begin{equation}
        \label{eq:f_contraction_condition_2}
        k = 1 - 
        \sqrt{ 
        1 
        -
        \frac{4 C_2 (d_v v_* \|\mathcal{L}\mathds{1}_\Omega\|_\infty + \epsilon)}{(m - \beta_1 d_v)^2}
        }.
    \end{equation}
    Since the discriminant is positive, the square root is real and positive, guaranteeing $k < 1$. Thus, $\mathcal{F}$ is a contraction on $B_{R_*}$. By the Banach Fixed Point Theorem, there exists a unique solution $\varphi \in B_{R_*}$, completing the proof.
\end{proof}

Now, we show that the assumptions of Theorem \ref{thm:Stationary-Solutions-Existence} hold true in the case of the Klausmeier model, i.e.\ $f(v) = v^2W(v) - Bv$. We start with a detailed investigation of the uniform solutions $(v_*, w_*)$.

\begin{prp}
    \label{prp:ExistenceOfConstantSSKineticSystem}
    Let $A > 0$ and $B > 0$. The following statements hold true.
    \begin{itemize}
        \item If $A \geq 2B$ then there exist three or two real solutions of \eqref{EQ:ConstantSteadyStates}.
        \item If $A < 2B$ then there exists only one real solution of \eqref{EQ:ConstantSteadyStates}.
    \end{itemize}
\end{prp}

\begin{proof}
    From the second equation in \eqref{EQ:ConstantSteadyStates}, we get that $w = \frac{A}{v^2 + 1}$. Thus the first equation becomes
    \begin{equation}
        \label{EQ:StationaryPolynomial}
        -B v^3 + A v^2 - B v = 0.
    \end{equation}
    For all parameters $A > 0, B > 0$ we can find the solution of the following form
    \begin{align}
        \label{EQ:FirstSS}
        \left(v_{*,1}, w_{*,1}\right) &= \left(0, A\right).
    \end{align}
    Provided that $A > 2B$, the equation \eqref{EQ:StationaryPolynomial} admits two more solutions
    \begin{align}
        \left(v_{*,2}, w_{*,2}\right) &= \left(\frac{A - \sqrt{A^2 - 4B^2}}{2B}, \frac{2 B^2}{A - \sqrt{A^2 - 4B^2}}\right), \label{EQ:SecondSS}\\
        \left(v_{*,3}, w_{*,3}\right) &= \left(\frac{A + \sqrt{A^2 - 4B^2}}{2B}, \frac{2 B^2}{A + \sqrt{A^2 - 4B^2}}\right). \label{EQ:ThirdSS}
    \end{align}
    Finally, if $A = 2B$, then the solutions \eqref{EQ:SecondSS} and \eqref{EQ:ThirdSS} merge into 
    \begin{align}
        \label{EQ:FourthSS}
        \left( v_{*,4}, w_{*,4} \right) &= \left(1, B\right).
    \end{align}
    For $A < 2B$ there are no other real solutions of \eqref{EQ:StationaryPolynomial} than \eqref{EQ:FirstSS}.
\end{proof}

Now we conclude that there is only one uniform solution $(v_*, w_*)$ which satisfies the assumptions of Theorem \ref{thm:Stationary-Solutions-Existence}.

\begin{prp}
    \label{prp:nonlinearity_f_properties_at_v_star}
    Let $A > 2B$, and $W(v_*)$ be a stationary solution of the problem \eqref{eq:v_stationary}, \eqref{eq:v_bound_stationary}, with $v_* = v_{*, 3} > 1$ being a non-zero solution given by Proposition \ref{prp:ExistenceOfConstantSSKineticSystem}. Denote the non-linear function $f(v) = v^2W(v) - Bv$. Then for every $\epsilon > 0$ there exists $d_w > 0$, sufficiently small, and numbers $m > 0, M >0$ such that
    \begin{equation}
        \label{eq:nonlinearity_f_condition_existence}
        \|f(v_*)\|_\infty \leq \epsilon, 
        \quad
        f'(v_*) < -m
        \quad
        \text{ for all }
        x \in \Omega,
        \quad
        \|f''(v_*)\|_\infty < M.
    \end{equation}
\end{prp}

\begin{proof}
    The first inequality, $\|f(v_*)\|_\infty \leq \epsilon$, follows directly from Lemma \ref{lem:stationary_v_estimate_v_star} and the fact that the constant steady states satisfy $v_*^2w_* - Bv_* = 0$. We have:
    \begin{equation*}
        \|f(v_*)\|_\infty 
        =
        \|v_*^2W(v_*) - Bv_* - (v_*^2w_* - Bv_*)\|_\infty
        =
        v_*^2\|W(v_*) - w_*\|_\infty.
    \end{equation*}
    By Lemma \ref{lem:stationary_v_estimate_v_star}, for any $\epsilon' > 0$, we can choose $d_w > 0$ small enough such that $\|W(v_*) - w_*\|_\infty < \epsilon'$. Setting $\epsilon' = \epsilon / v_*^2$ yields the desired bound.

    Next, we analyze the Fr\'echet derivative of $f$. Since $v_*w_* = B$, the derivative evaluated at the constant function $v_*$ is given by
    \begin{align*}
        f'(v_*) 
        &=
        2v_*W(v_*) + v_*^2W'(v_*) - B \\
        &=
        2v_*(W(v_*) - w_*) + 2v_*w_* + v_*^2W'(v_*) - B \\
        &=
        2v_*(W(v_*) - w_*) + v_*^2W'(v_*) + B.
    \end{align*}
    We analyze the terms in the limit as $d_w \to 0$. The first term vanishes, since $\|W(v_*) - w_*\|_\infty \to 0$. The operator $W'(v_*)$ is the Fr\'echet derivative of the map $v \mapsto W(v)$, which is the solution operator of the water equation $d_w \Delta W - (v^2+1)W + A = 0$. By differentiating this equation with respect to $v$ in the direction of a perturbation $h$, we find that $W'(v_*)h$ is the solution to
    \begin{equation*}
        (d_w \Delta - (v_*^2+1)\mathrm{Id})(W'(v_*)h) = 2v_* W(v_*) h.
    \end{equation*}
    In the limit $d_w \to 0$, this becomes an algebraic relation: $-(v_*^2+1)W'(v_*)h = 2v_* w_* h$. Thus, the operator $W'(v_*)$ converges in the operator norm to multiplication by the constant $-\frac{2v_*w_*}{v_*^2+1} = -\frac{2B}{v_*^2+1}$.
    
    Consequently, the derivative $f'(v_*)$ converges to a constant value as $d_w \to 0$:
    \begin{equation*}
        \lim_{d_w \to 0} f'(v_*) = 0 + v_*^2 \left( -\frac{2B}{v_*^2+1} \right) + B = B \left( 1 - \frac{2v_*^2}{v_*^2+1} \right) = B \frac{1-v_*^2}{1+v_*^2}.
    \end{equation*}
    As stated in the proposition, we consider the steady state $v_* = v_{*,3} > 1$, which implies $1 - v_*^2 < 0$. Therefore, the limit of $f'(v_*)$ is a strictly negative constant. Let this limit be $L < 0$. By the continuity of the derivative with respect to $d_w$, for any small $\delta > 0$, we can choose $d_w$ sufficiently small such that $f'(v_*) < L + \delta$. By choosing $\delta < -L/2$, we ensure that $f'(v_*) < L/2 < 0$. We can thus define a positive constant $m = -L/2$, satisfying $f'(v_*) < -m < 0$ for all $x \in \Omega$.

    Finally, we verify that the second Fr\'echet derivative of $f$ is bounded. The second derivative, $f''(v_*)$, is a bilinear operator that acts on a pair of perturbations $(h, k)$. Differentiating $f'(v)[k]$ with respect to $v$ in the direction $h$, and evaluating at $v=v_*$, yields:
    \begin{equation*}
        f''(v_*)[h, k] = 2W(v_*)hk + 2v_*\left( (W'(v_*)[h])k + (W'(v_*)[k])h \right) + v_*^2(W''(v_*)[h,k]).
    \end{equation*}
    To find $W''(v_*)$, we differentiate the equation for $W'(v)[k]$ once more:
    \begin{equation*}
        (d_w \Delta - (v_*^2+1)\mathrm{Id})(W''(v_*)[h, k]) = 2W(v_*)hk + 2v_*\left( (W'(v_*)[h])k + (W'(v_*)[k])h \right).
    \end{equation*}
    Since the operator on the left-hand side is invertible and its inverse is bounded, $W''(v_*)$ is a bounded bilinear operator. As $\|W(v_*)\|_\infty$ and $\|W'(v_*)\|_\infty$ are also bounded, it follows from the triangle inequality that $\|f''(v_*)\|_\infty$ is bounded. We can therefore define a constant $M$ as
    \begin{equation*}
        M := 2\|W(v_*)\|_\infty + 4v_*\|W'(v_*)\|_\infty + v_*^2\|W''(v_*)\|_\infty,
    \end{equation*}
    which is finite and positive. This completes the proof.
\end{proof}

We continue by analyzing stability of the stationary solutions obtained in this section.

\section{Stability analysis}

We derive the conditions under which the vegetation will survive or go extinct.

\subsection{Extinction of vegetation} 

We show that a sufficiently small initial density concentration leads to the extinction of vegetation.

\begin{thm}
    \label{thm:Exinction-Of-Vegetation}
    Let $M = \max\{\|v_0\|_\infty, A\}$. Denote an interval $\Gamma  = \left[0, \frac{B}{M}\right]$. Suppose that $v_0(x) \in \Gamma$ for all $x \in \Omega$. Then the solution $(v, w)$ of the problem \eqref{eq:u}--\eqref{eq:v_init} satisfies
    \begin{equation}
        v(x,t) \xrightarrow{t \rightarrow \infty} 0, \quad w(x,t) \xrightarrow{t \rightarrow \infty} W_0(x),
    \end{equation}
    where we denoted the stationary water profile $W_0 := W(0)$ from Lemma \ref{lem:existence_of_inhomogeneous_elliptic_v}.
\end{thm}
\begin{proof}
    The first limit is a consequence of the bound \eqref{EQ:AprioriEsitmate_Unvariant_Region}. Since $(B - M\nu(0)) \geq 0$ we get that for all $x \in \Omega$
    \begin{equation*}
        v(x,t) \leq \frac{B\nu(0)}{M \nu(0) + 
        (B - M \nu(0)) e^{B t}} \rightarrow 0 \quad \text{as} \quad t \rightarrow \infty.
    \end{equation*}
    Now, let us replace the quadratic term in equation \eqref{eq:v_stationary} with zero to get
    \begin{equation}
        \label{eq:v_elliptic_linear_part}
        d_w \Delta W - W + A = 0, \quad x \in \Omega, \quad W = 0, \quad x \in \partial \Omega.
    \end{equation}
    Recall that Lemma \ref{lem:existence_of_inhomogeneous_elliptic_v} yields existence of a non-constant solution $W_0$ of the problem \eqref{eq:v_elliptic_linear_part}. Denote $\psi(x,t) = w(x,t) - W_0(x)$. One can easily verify that it satisfies the following partial differential equation.
    \begin{align}
        \label{eq:w_parabolic}
        \begin{cases}
            \psi_t = d_w\Delta \psi - \psi - v^2w &\quad t > 0, x \in \Omega, \\
            \psi = 0 &\quad t > 0, x \in \partial\Omega, \\
            \psi(0) = w_0 - W_0 &\quad x \in \Omega.
        \end{cases}
    \end{align}
    Denote by $e^{t(d_w\Delta - \mathrm{Id})}$ the heat semigroup, then apply the Duhamel principle to equation \eqref{eq:w_parabolic}
    \begin{equation}
        \label{eq:w_duhmel_principle}
        \psi(t) = e^{t(d_w\Delta - \mathrm{Id})} \psi(0) - \int_0^t e^{(t-s)(d_w\Delta - \mathrm{Id})} v^2(s) w(s)\:ds.
    \end{equation}
    Equation \eqref{eq:w_duhmel_principle} allows us to estimate the $L^\infty$ norm of $u$.
    \begin{align}
        \label{eq:w_infty_estimate}
        \begin{split}
            \|\psi(t)\|_\infty 
            &\leq e^{-t}\|\psi(0)\|_\infty + \int_0^t e^{-(t-s)} \|v^2(s)\|_\infty \|w(s)\|_\infty\:ds \\
            &\leq e^{-t}\|\psi(0)\|_\infty + M \int_0^t e^{-(t-s)} \|v^2(s)\|_\infty \:ds.
        \end{split}
    \end{align}
    Now, we use inequality \eqref{EQ:AprioriEsitmate_Unvariant_Region} to estimate the integral term as follows
    \begin{align}
        \label{eq:w_infty_estimate_integral}
        \begin{split}
            M \int_0^t e^{-(t-s)} \|v^2(s)\|_\infty \:ds 
            &\leq
            \frac{M B \nu(0)}{B-M \nu(0)} \int_0^t e^{-(t-s)} e^{-Bs}\:ds \\
            &= 
            \frac{M B \nu(0)}{(B-M \nu(0))(B-1)} e^{-t} 
            - 
            \frac{M B \nu(0)}{(B - M \nu(0))(B-1)} e^{-Bt}.
        \end{split}
    \end{align}
    Combining inequalities \eqref{eq:w_infty_estimate} and \eqref{eq:w_infty_estimate_integral} yields
    \begin{align*}
        \|\psi(t)\|_\infty 
        &\leq \left( \|\psi(0)\|_\infty + \frac{M B \nu(0)}{(B-M\nu(0))(B-1)} \right) e^{-t} \\
        &- \frac{M B \nu(0)}{(B-M\nu(0))(B-1)} e^{-Bt}.
    \end{align*}
    Thus, it follows that $\|\psi(t)\|_\infty \rightarrow 0$ as $t \rightarrow \infty$.
\end{proof}

Roughly speaking Theorem~\ref{thm:Exinction-Of-Vegetation} states that zero vegetation is a stable solution, while non-zero vegetation is unstable for too small initial vegetation. We continue by addressing the question under which conditions the vegetation can persist. 

\subsection{Survival of vegetation}

We show the asymptotic stability of the non-zero stationary solutions, ensuring the survival of vegetation.

\begin{thm}
    \label{thm:Stationary-Solutions-Stability}
    Let the parameters $A > 0, B >0$ satisfy $A > 2B$. Let $(v_*, w_*)$ be the constant solutions from Proposition \ref{prp:ExistenceOfConstantSSKineticSystem}, satisfying $v_* > 1$. We denote by $(V, W)$ the stationary solution contained in a neighborhood $B_R(v_*, w_*)$ of radius $R>0$, given by Theorem \ref{thm:Stationary-Solutions-Existence}. Then, whenever the dispersal and diffusion rates $d_v, d_w$ are sufficiently small, the solution $(V, W)$ is linearly, exponentially stable in $L^2(\Omega) \times L^2(\Omega)$.
\end{thm}

\begin{proof}
    Let $(V, W)$ be a solution constructed in Theorem \ref{thm:Stationary-Solutions-Existence}. We set $f(v) = v^2W(v) - Bv$ as in Proposition \ref{prp:nonlinearity_f_properties_at_v_star}. Then, we linearize the vegetation problem \eqref{eq:u}, \eqref{eq:u_bound} and \eqref{eq:Unit} around $V$ and introduce a peturbation $\varphi(x,t) = V(x) - v(x,t)$. This new variable satisfies the following linearized problem
    \begin{numcases}{}
        \varphi_t = d_v\mathcal{L} \varphi + f'(V)\varphi & \hspace{2mm} $(x, t) \in \Omega \times [0, \infty],$ \label{eq:varphi} \\
        \varphi = 0 & \hspace{2mm} $(x,t) \in \mathbb{R}^n \setminus \Omega  \times [0,\infty], $ \label{eq:varphi_bound} \\
        \varphi(0) = V - v_0 & \hspace{2mm} $x \in \Omega .$ \label{eq:varphi_init}
    \end{numcases}
    Let $\beta_1 > 0$ be the principle eigenvalue of $-\mathcal{L}$, applying the standard energy method, we get that
    \begin{equation}
        \label{eq:v_energy_estimate_1}
        \frac{1}{2}\frac{d}{dt} \int_\Omega \varphi^2 \:dx
        \leq
        -d_v \beta_1 \int_\Omega \varphi^2\:dx + \int_\Omega f'(V) \varphi^2\:dx.
    \end{equation}
    Assuming that $d_w$ is sufficiently small, we deduce by continuity argument and Proposition \ref{prp:nonlinearity_f_properties_at_v_star} that the second integral is negative, hence inequality \eqref{eq:v_energy_estimate_1} yields
    \begin{equation}
        \label{eq:v_energy_estimate_2}
        \frac{1}{2}\frac{d}{dt} \int_\Omega \varphi^2 \:dx
        \leq
        -d_v \beta_1 \int_\Omega \varphi^2\:dx.
    \end{equation}
    The differential inequality \eqref{eq:v_energy_estimate_2}, when solved gives the estimate
    \begin{equation}
        \label{eq:v_energy_estimate_3}
        \int_\Omega \varphi^2 \:dx
        \leq
        \int_\Omega(V-v_0)^2\:dx\:  e^{-2 d_v \beta_1 t} =: M_0 e^{-2 d_v \beta_1 t}.
    \end{equation}    
    Hence $V$ is linearly exponentially stable. Now, consider a perturbation of the solution $W$, given by $\psi(x,t) = W(x) - w(x,t)$. The corresponding evolution problem is given by
    \begin{numcases}{}
        \psi_t = d_w\Delta \psi - \psi -(V^2 W - v^2 w) & \hspace{2mm} $(x, t) \in \Omega \times [0, \infty],$ \label{eq:psi} \\
        \psi = 0 & \hspace{2mm} $(x,t) \in \partial \Omega  \times [0,\infty], $ \label{eq:psi_bound} \\
        \psi(0) = W - w_0 & \hspace{2mm} $x \in \Omega .$ \label{eq:psi_init}
    \end{numcases}
    First, let us re-write the non-linear term appearing in equation \eqref{eq:psi}
    \begin{equation}
        \label{eq:w_non-linear_term}
        V^2W - v^2w 
        =
        (V - v)(V + v) W + v^2 (W - w)
        =
        \varphi (V + v) W + v^2 \psi.
    \end{equation}
    Let $\lambda_1 > 0$ be the principle eigenvalue of $-\Delta$, the energy method yields the following estimate
    \begin{equation}
        \label{eq:w_energy_estimate_1}
        \frac{1}{2}\frac{d}{dt}\int_\Omega \psi^2\:dx
        \leq
        - d_w \lambda_1 \int_\Omega \psi^2\:dx - \int_\Omega (1 + v^2) \psi^2\:dx - \int_\Omega W (V + v) \varphi \psi \:dx
    \end{equation}
    Since $v, V, W \in L^\infty(\Omega)$ and $v$ is global solution, we set $C = \sup\{\|W(V+v)\|_\infty : t > 0\}$. Applying $\epsilon$-Cauchy inequality $ab \leq \frac{\epsilon}{2}a^2 + \frac{1}{2\epsilon} b^2$ to the last term we obtain and estimating with inequality \eqref{eq:v_energy_estimate_3} we obtain
    \begin{equation}
        \label{eq:w_energy_estimate_2}
        \int_\Omega W(V+v) \varphi \psi \:dx
        \leq
        \frac{C \epsilon}{2}\int_\Omega \psi^2 \:dx + \frac{C}{2 \epsilon} \int_\Omega \varphi^2 \:dx 
        \leq
        \frac{C \epsilon}{2}\int_\Omega \psi^2 \:dx + \frac{C M_0}{2 \epsilon} e^{-2d_v\beta_1 t}
    \end{equation}
    Combining inequalities \eqref{eq:w_energy_estimate_1} and \eqref{eq:w_energy_estimate_2} we get the differential inequality
    \begin{equation}
        \label{eq:w_energy_estimate_3}
        \frac{1}{2}\frac{d}{dt}\int_\Omega \psi^2 \:dx
        \leq
        \left( -d_w \lambda_1 - 1 + \frac{C\epsilon}{2} \right) \int_\Omega \psi^2\:dx + \frac{C M_0}{2\epsilon} e^{-2d_v \beta_1 t}
    \end{equation}
    Solving differential inequality \eqref{eq:w_energy_estimate_3} we obtain
    \begin{equation}
        \label{eq:w_energy_estimate_4}
        \int_\Omega \psi^2 \:dx \leq \left( \int_\Omega (W-w_0)^2 \:dx \right) e^{-\mu t} + \frac{C M_0}{\epsilon (\mu - 2d_v \beta_1)} \left( e^{-2d_v \beta_1 t} - e^{-\mu t} \right),
    \end{equation}
    where we denoted constant $\mu = 2(d_w\lambda_1 + 1) - C\epsilon$, which is positive provided that $\epsilon < \frac{2(d_w \lambda_1 + 1)}{C}$. Since choice of $\epsilon > 0$ was arbitrary, we proved that $W$ is exponentially stable.
\end{proof}

Thus far, we have derived the asymptotic behavior through theoretical analysis. We now proceed to validate and extend these results using numerical experiments.

\section{Numerical experiments}
We have designed numerical experiments addressing our main results in Theorems \ref{thm:Stationary-Solutions-Existence}, \ref{thm:Exinction-Of-Vegetation}, and \ref{thm:Stationary-Solutions-Stability}, which summarize as follows: 

The non-zero, constant vegetation and water profiles $(v_*, w_*)$ from Proposition \ref{prp:ExistenceOfConstantSSKineticSystem} can be perturbed in such a way that the stationary problem \eqref{eq:u}--\eqref{eq:v_init} admits a non-zero, linearly, asymptotically stable stationary solution $(V, W)$ in its neighborhood, whenever: 
\begin{itemize}
    \item[(A1)] The dispersal rates $d_v > 0, d_w > 0$ are sufficiently small, matching the assumptions of Theorems \ref{thm:Stationary-Solutions-Existence} and \ref{thm:Stationary-Solutions-Stability}.
    \item[(A2)] We assume the dispersal rate $d_v > 0$ and domain $\Omega$ to be such that the principal eigenvalue $\beta_1$ of $\mathcal{L}$ satisfies \eqref{eq:d_v_condition_1}. Note that by Remark \ref{rem:dispersal_patch_condition} this means that $\Omega$ is sufficiently large.
    \item[(A3)] The parameters $A > 2B$, $B > 0$ of the Klausmeier model \eqref{eq:u}--\eqref{eq:v_init} and the constant solutions $(v_*, w_*)$ from Proposition \ref{prp:ExistenceOfConstantSSKineticSystem} satisfying $v_* > 1$.
\end{itemize}
For demonstrations in the one-dimensional case, we consider the intervals $\Omega = (-L, L)$ with suitable $L > 0$ and eigenfunctions of $\Delta$ with Dirichlet boundary condition given by $\phi_j(x) = \frac{1}{\sqrt{L}}\sin(\frac{j\pi x}{L})$, which are uniformly bounded by the constant $C = \frac{1}{\sqrt{L}}$. Hence assumptions \ref{O1}--\ref{O2} are satisfied. Further, we model the dispersal kernel $J$ by sub- and super-Gaussian densities 
\begin{equation}
    \label{eq:1dGaussian_PDF}
    J_1(x) = \frac{1}{\sqrt{2}} e^{-\sqrt{2}|x|}, \quad
    J_2(x) = 2 \left(\frac{\Gamma(3/4)^{1/2}}{\Gamma(1/4)^{3/2}}\right) e^{- \left( \frac{\Gamma(3/4)}{\Gamma(1/4)}\right)^2 x^4}, \quad x \in \mathbb{R}.
\end{equation}
Kernels $J_1$ and $J_2$ satisfy assumptions \ref{J1}--\ref{J5}. We shall compare our results to the classical reaction-diffusion system. In order to do that properly, we examine the Taylor expansion of the dispersal operator $\mathcal{L}$ with given kernel $J$:
\begin{equation}
    \label{eq:dispersal_op_taylor_expansion}
    \mathcal{L} v = \frac{1}{2}\left(\int_{\mathbb{R}} J(z)z^2dz\right) \Delta v + \text{h.o.t.}
\end{equation}
With our choice of the kernels $J_1, J_2$ we have that their variances are equal and satisfy the equation
\begin{equation}
    \label{eq:experiment_kernels_equal_variance}
    \int_{\mathbb{R}} J_1(z)z^2 dz = \int_{\mathbb{R}} J_2(z)z^2 dz = 1.
\end{equation}
Consequently, we use the local counterpart of the non-local model \eqref{eq:u}--\eqref{eq:v_init} to examine the effect of the higher-order terms (h.o.t.), given by
\begin{numcases}{}
    v_t = \frac{d_v}{2} \Delta v + v^2 w - B v & \hspace{2mm} $(x, t) \in \Omega \times [0, T],$ \label{eq:local_u} \\
    w_t = d_w\Delta w - v^2w - w + A & \hspace{2mm} $(x, t) \in \Omega \times [0, T],$ \label{eq:local_v} \\
    v = 0 & \hspace{2mm} $(x,t) \in \partial\Omega \times [0,T], $ \label{eq:local_u_bound} \\
    w = 0 & \hspace{2mm} $(x,t) \in \partial\Omega \times [0,T],$ \label{eq:local_v_bound} \\
    v(0) = v_0 & \hspace{2mm} $x \in \Omega ,$ \label{eq:local_u_init} \\
    w(0) = w_0 & \hspace{2mm} $x \in \Omega\,, $ \label{eq:local_v_init}
\end{numcases}

Now, we are ready to discuss the numerical experiments. We begin with the examination of the impact of the patch half-width $L$ on the stationary solutions.

\begin{expe}
    We design a numerical experiment to investigate the existence of a critical patch size for vegetation survival, a concept discussed in our theoretical analysis (see Remark \ref{rem:dispersal_patch_condition}). The experiment's primary goal is to compare the critical patch size of the classical local (reaction-diffusion) model against two distinct non-local models. These non-local models employ dispersal kernels with ``thin tails'' (super-Gaussian) and ``fat tails'' (sub-Gaussian), respectively. Crucially, both non-local kernels are calibrated to have the same variance (second moment), ensuring that their leading-order Taylor approximations are equivalent to the same Laplacian operator. This setup allows us to directly isolate and observe the influence of higher-order terms, which represent the shape of the dispersal kernel, on population persistence.
    
    The simulation framework is configured as follows:
    \begin{itemize}
        \item By Proposition \ref{prp:ExistenceOfConstantSSKineticSystem}, we assume that $A > 2B$ to guarantee the existence of non-zero, spatially uniform vegetation and water solutions $(v_*, w_*)$ of the system \eqref{EQ:ConstantSteadyStates}. In particular, we take mortality $B = 0.45$ and rainfall $A = 1.8$. For this choice of parameters we get that $(v_{*, 3}, w_{*, 3})$ satisfies (A3).
        \item We consider fixed dispersal rate $d_v = 2.0$ and diffusion rate $d_w = 0.1$.
        \item We choose the range of patch half-widths $L \in [10^{0}, 10^2]$. We split this interval into $50$ points equally distributed along the logarithmic scale.
    \end{itemize}
    For each value of $L$, we find the stationary solution for all three models by simulating the system's dynamics. The simulation is performed using a forward Euler time-stepping scheme with a time step of $h_t = 10^{-4}$. The system is considered to have reached a steady state when the $\ell^2$-norm of the difference between consecutive time steps falls below a tolerance of $\epsilon_{\text{tol}} = 10^{-5}$. To encourage convergence to non-trivial solutions, the initial condition for each simulation is given by the spatially uniform steady state $(v_{*,3}, w_{*,3})$, which is then perturbed by a small cosine profile. The spatial domain is discretized using a number of equidistant nodes $N$ that increases proportionally with $L$ to maintain adequate numerical resolution across all scales.

    To quantify the results, we record the average biomass density of the final stationary solution for each $L$. The critical patch size, $L_{\text{crit}}$, is then numerically determined for each model as the largest patch half-width for which the average biomass falls below a small threshold of $0.1$.
\end{expe}
We present the results of our numerical simulations in the figure below.
\begin{figure}[H]
    \centering
    \includegraphics[width=\textwidth]{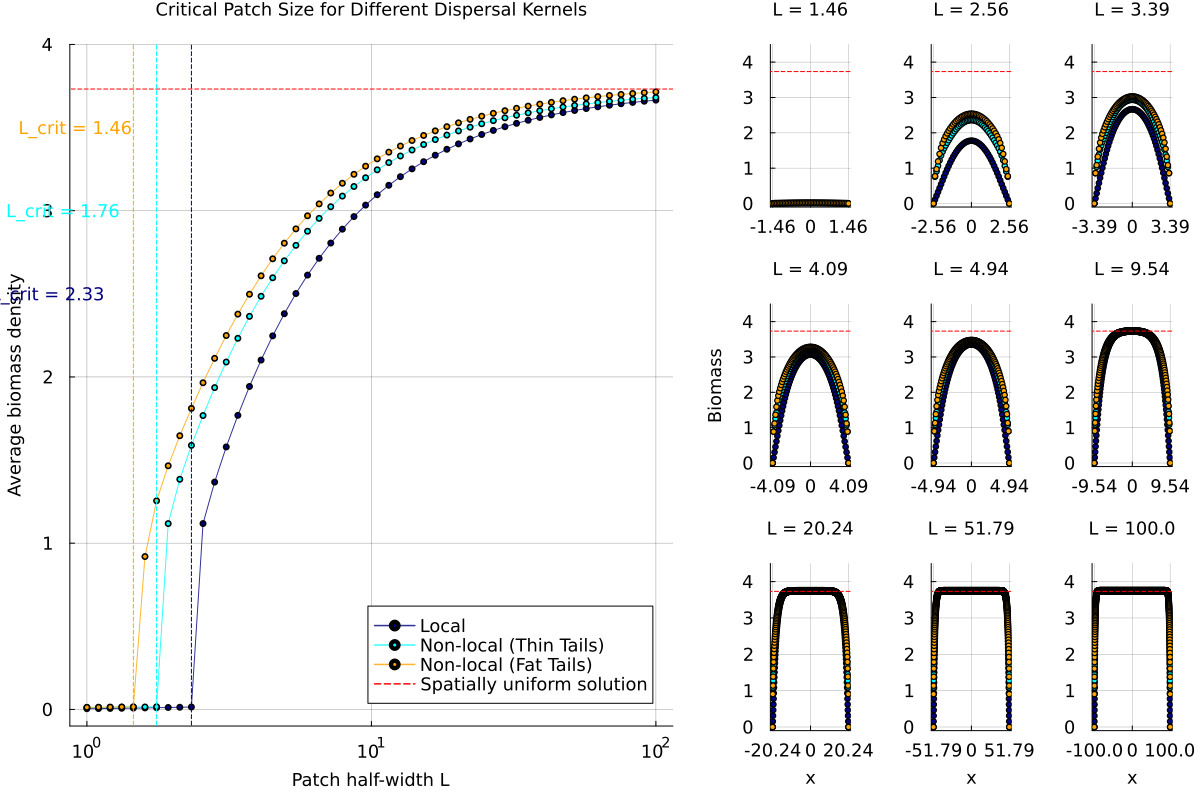}
    \caption{In the left panel, we show the spatially uniform steady-state solution $v_{*,3}$ (red dashed line), the average stationary biomass in the non-local models with thin and fat tails (cyan and orange, respectively) and in the local model (blue) as a function of the patch half-width $L$. We also present the corresponding, numerically estimated critical patch sizes (dashed vertical lines). In the right panel, we present a gallery of stationary biomass density profiles (orange, cyan and blue dots) for different patch half-widths $L$, corresponding to the results shown in the left panel.}
    \label{fig:critical_patch_size}
\end{figure}
\begin{con}
As shown in the left panel of Figure \ref{fig:critical_patch_size}, all three models—local, non-local with thin tails, and non-local with fat tails—exhibit the same qualitative behavior for large domains. Specifically, their respective stationary solutions converge to the spatially uniform steady state $v_{*, 3}$ as the patch half-width $L$ increases and boundary effects become negligible. This agrees with our existence and stability results contained in Theorems \ref{thm:Stationary-Solutions-Existence} and \ref{thm:Stationary-Solutions-Stability}.
\end{con}

\begin{con}
The left panel of Figure \ref{fig:critical_patch_size} reveals that, while all models eventually collapse to the desert state, the critical patch size required for survival is different for each. The model with the fat-tailed (sub-Gaussian) kernel is the most resilient, exhibiting the smallest critical patch size ($L_{\text{crit}} \approx 1.46$). This is followed by the thin-tailed (super-Gaussian) kernel ($L_{\text{crit}} \approx 1.76$), while the local model is the least resilient, requiring the largest habitat for survival ($L_{\text{crit}} \approx 2.33$). Since all non-local kernels were calibrated to have the same variance, this ordering is directly attributable to the influence of higher-order moments. This provides evidence that long-range dispersal events, which are more frequent in kernels with a larger fourth moment (kurtosis), significantly enhance the ability of vegetation to persist in smaller, fragmented habitats.
\end{con}

\begin{con}
The right panel of Figure \ref{fig:critical_patch_size} shows a difference in the solution profiles at the domain boundary. The local model's solution, governed by diffusion, is continuous and smoothly approaches zero. In contrast, both non-local models exhibit a sharp drop-off at the boundary. This is a direct consequence of Theorem \ref{thm:Stationary-Solutions-Existence}, namely equation \eqref{eq:f_ball_to_ball_radius} defining the distance $R_*$ between the constant value $v_{*, 3}$ and the stationary solution $V$. Note that for small diffusion and dispersal rates $d_v, d_w$ it is true that $R_* < v_{*, 3}$, directly implying discontinuity at the boundary.
\end{con}

\begin{expe}
    We propose an experiment in which we investigate the existence of a critical level of concentration of the vegetation as in Theorem \ref{thm:Exinction-Of-Vegetation}. We proceed by computing the bifurcation diagram for the local and non-local Klausmeier model. Having that, we will check if the branches are bounded from below by $\sim\frac{B}{A}$. In the regime of decreasing rainfall $A$ the stationary vegetation $V$ obtained by Theorems \ref{thm:Stationary-Solutions-Existence} and \ref{thm:Stationary-Solutions-Stability} ceases to exist once $A = 2B$. We also want to validate whether there exists a regime of parameters $d_v, d_w$ such that inhomogeneous solutions emerge. We assume the following framework of the experiment:
    \begin{itemize}
        \item We vary the rainfall $A \in [0.1, 3.0]$ with fixed mortality $B = 0.45$.
        \item We consider the dispersal rate $d_v = 2.0$ and diffusion rate $d_w \in \{0.1, 80.0\}$.
        \item We take the patch half-width $L = 25$.
    \end{itemize}
    The bifurcation diagrams were generated using the Julia package \texttt{BifurcationKit.jl}. The numerical method relies on a Pseudo-Arclength Continuation (PALC) algorithm to trace solution branches. For the spatial discretization, the domain $\Omega = (-L, L)$ is discretized using $N = \lfloor 3L \rfloor$ equidistant nodes. The Laplacian operator $\Delta$ is approximated using a standard second-order centered finite difference scheme, and the non-local integral operator $\mathcal{L}$ is approximated using the composite trapezoidal rule. The convergence tolerance $\epsilon$ for the underlying Newton solver was set to $\epsilon = 10^{-10}$.

    The continuation process was initiated from the parameter value $A = 3.0$. The initial guess was constructed from the spatially homogeneous steady state $(v_0, w_0) = (v_{*, 3}, w_{*, 3})$, as given by Proposition \ref{prp:ExistenceOfConstantSSKineticSystem}, with a small cosine-shaped spatial perturbation applied to encourage convergence to non-uniform solution branches.
\end{expe}

We present the results of numerical simulations done in the regime of slow diffusion, e.g.\ with $d_w = 0.1$.

\begin{figure}[H]
    \centering
    \includegraphics[width=\textwidth]{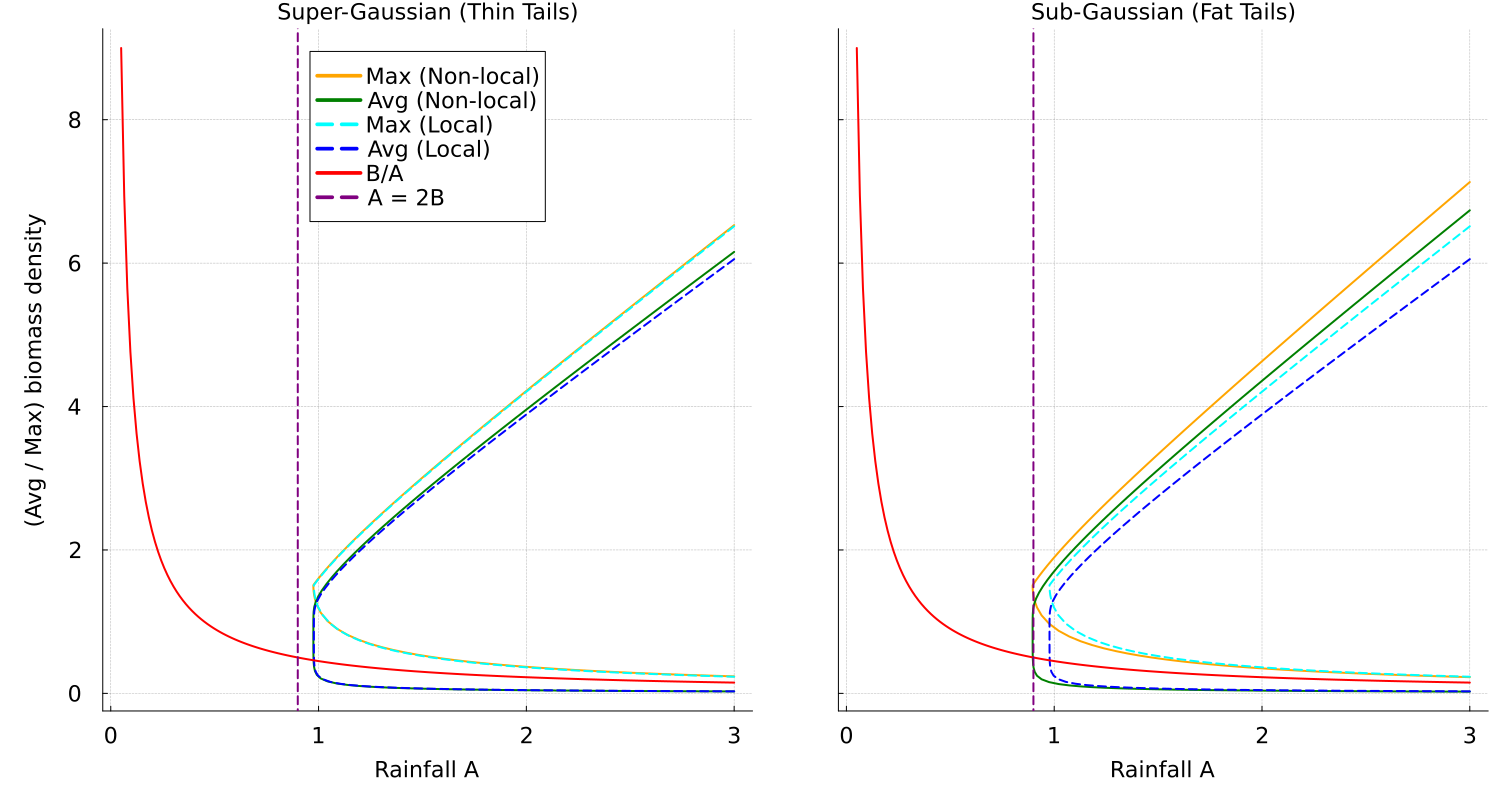}
    \caption{Comparison of bifurcation diagrams for the non-local models using a super-Gaussian (thin tails) and a sub-Gaussian (fat tails) dispersal kernel, both shown against the standard local model. Solid orange and green lines represent the maximum and average biomass for the non-local models, while dashed cyan and blue lines show the same for the local model. We mark the critical biomass threshold (red line, $\sim B/A$) and the critical rainfall for the kinetic system (purple line, $A=2B$). Here $d_v = 2.0, d_w = 0.1, B = 0.45$.}
    \label{fig:bifurcation_diagram_slow}
\end{figure}

\begin{con}
    The results shown in Figure \ref{fig:bifurcation_diagram_slow} align with our theoretical results from Theorem \ref{thm:Exinction-Of-Vegetation}. In particular, every stationary solution has the maximal value bounded from below by the value $\sim B/A$.
\end{con}

\begin{con}
    Our simulation confirms the statement of Theorems \ref{thm:Stationary-Solutions-Existence}, \ref{thm:Stationary-Solutions-Stability}, namely the existence of a stable, spatially perturbed steady state up to the point when the rainfall $A$ approaches the value $A = 2B$. When $A < 2B$, the spatially uniform solutions $(v_{*, 3}, w_{*, 3})$ cease to exist and the vegetation goes extinct.
\end{con}

Finally, we show the results of numerical simulations done in the regime of fast diffusion, e.g.\ with $d_w = 80.0$.

\begin{figure}[H]
    \centering
    \includegraphics[width=\textwidth]{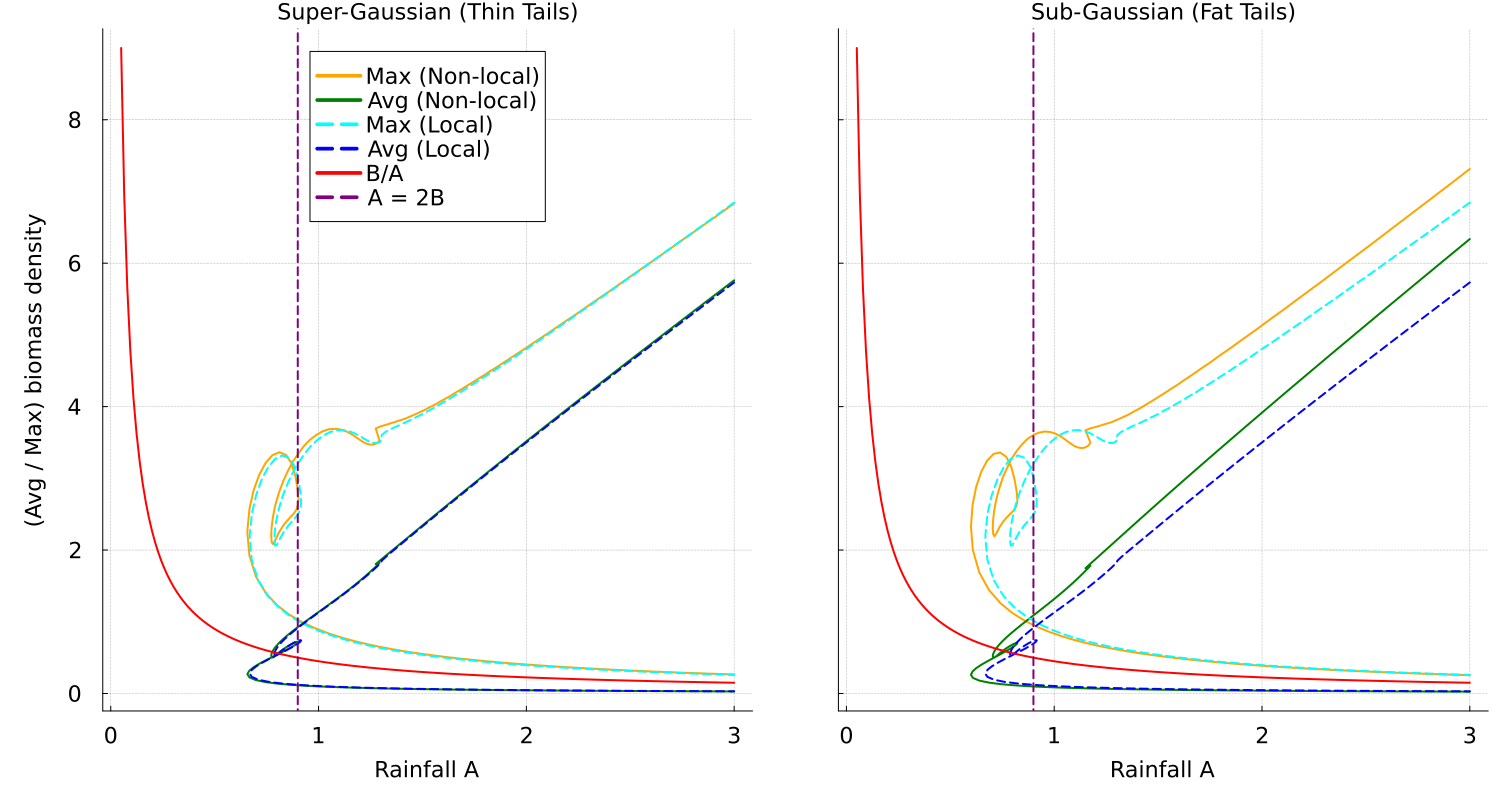}
    \caption{Comparison of bifurcation diagrams for the non-local models using a super-Gaussian (thin tails) and a sub-Gaussian (fat tails) dispersal kernel, both shown against the standard local model. Solid orange and green lines represent the maximum and average biomass for the non-local models, while dashed cyan and blue lines show the same for the local model. We mark the critical biomass threshold (red line, $\sim B/A$) and the critical rainfall for the kinetic system (purple line, $A=2B$). Here $d_v = 2.0, d_w = 80.0, B = 0.45$.}
    \label{fig:bifurcation_diagram_fast}
\end{figure}

\begin{con}
    A comparison of the bifurcation diagrams in Figure \ref{fig:bifurcation_diagram_fast} reveals the impact of the dispersal kernel's shape on the system's dynamics. While both the super-Gaussian (thin tails) and sub-Gaussian (fat tails) models support complex, looping solution branches, the sub-Gaussian kernel yields the existence of non-deserted solutions over a wider range of the rainfall parameter $A$. The oscillatory behavior suggests a cascade of secondary bifurcations. Notably, in both non-local scenarios, patterned solutions emerge for rainfall levels below the kinetic system's threshold ($A < 2B$).
\end{con}

\begin{con}
    Across all models presented in Figures \ref{fig:bifurcation_diagram_slow} and \ref{fig:bifurcation_diagram_fast}, the curve $\sim B/A$ serves as a sharp lower bound for the biomass of all non-trivial stationary solutions. This provides robust numerical validation for the vegetation extinction threshold predicted by Theorem \ref{thm:Exinction-Of-Vegetation}, confirming its applicability to both local and a range of non-local dispersal strategies.
\end{con}

For the sake of exposition we provide the spatial visualization of patterns obtained in Figure \ref{fig:bifurcation_diagram_fast}.

\begin{figure}[H]
    \centering
    \includegraphics[width=\textwidth]{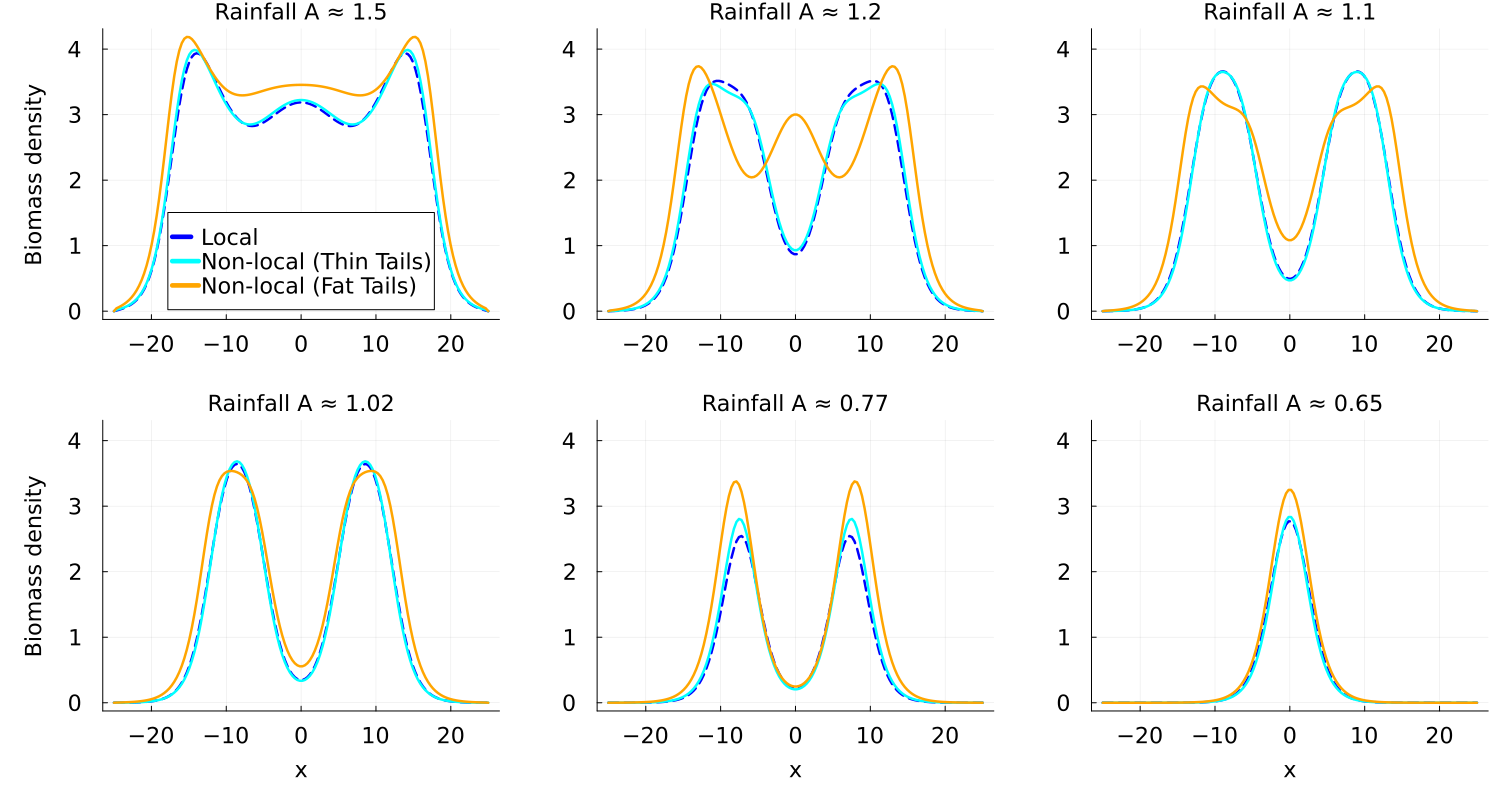}
    \caption{A gallery of stationary biomass densities for non-local and local models (green and blue dashed lines respectively). We selected several points from the upper branch of the rightmost bifurcation diagram in Figure \ref{fig:bifurcation_diagram_fast}. Here $d_v =2.0, d_w = 80.0, B = 0.45, L=25$.}
    \label{fig:bifurcation_profiles_fast}
\end{figure}

\begin{con}
    The bifurcation diagram for the fat-tailed kernel $J_2$ with patch half-width $L=25$ shown in Figure \ref{fig:bifurcation_diagram_fast} gives very similar loopy structures in both models. However, the vegetation profiles corresponding to this diagram, presented in Figure \ref{fig:bifurcation_profiles_fast}, do not always share the same geometry with respect to the local / non-local model, particularly for $A = 1.2$. It is also worth noting that, due to the high diffusion regime, the characteristic sharp drop-off at the boundary has disappeared.
\end{con}

\subsection{Final conclusions and future work}

In this work, we have established a rigorous mathematical framework for the Klausmeier model with non-local dispersal \eqref{eq:u}--\eqref{eq:v_init} on finite flat landscapes, proving the existence and stability of non-zero stationary solutions (Theorems \ref{thm:Stationary-Solutions-Existence} and \ref{thm:Stationary-Solutions-Stability}) and deriving conditions for vegetation extinction (Theorem \ref{thm:Exinction-Of-Vegetation}). 

Unexpectedly, our numerical experiments revealed that non-local models can sustain vegetation on significantly smaller habitats compared to the classical reaction-diffusion counterpart \eqref{eq:local_u}--\eqref{eq:local_v_init}, as illustrated in Figure \ref{fig:critical_patch_size}.
This finding challenges the intuitive logic that non-local dispersal, by facilitating long-range transport, should increase the flux of biomass through the boundaries into the uninhabitable desert, thereby making extinction more likely.

We attempted to numerically investigate the underlying causes of this enhanced resilience using the theoretical criteria derived in this paper, specifically the spectral condition in Lemma \ref{lem:apriori_smallnes_d_v} (discussed in Remark \ref{rem:dispersal_patch_condition}) and the expression for the critical radius given by formula \eqref{eq:f_ball_to_ball_radius}. 
However, due to the highly implicit nature of these conditions—particularly the complex dependence of the principal eigenvalue and the perturbation radius on the specific kernel shape—we were unable to isolate a simple analytical reason for why the critical patch size decreases for more non-local operators.

In future work, we intend to investigate this problem much further through both extensive numerical experimentation and theoretical analysis. 
This will also include providing a precise proof for the existence of non-constant, patterned solutions in the regime $A < 2B$, where our numerical results (see Figure \ref{fig:bifurcation_diagram_fast}) suggest vegetation persists even after the uniform solution has collapsed.

\newpage

\bibliographystyle{siam}
\bibliography{bibliography}

@book{Arendt2001,
      author        = "Arendt, Wolfgang and Batty, Charles J K and Hieber,
                       Matthias and Neubrander, Frank",
      title         = "{Vector-valued Laplace transforms and Cauchy problems}",
      publisher     = "Birkhäuser",
      address       = "Basel",
      series        = "Monographs in mathematics",
      year          = "2001",
      url           = "https://cds.cern.ch/record/517338",
      doi           = "10.1007/978-3-0348-5075-9",
}

@article{BaiLiWang2025,
title = {Existence and nonexistence of stable patterns in semilinear nonlocal diffusion equations},
journal = {Journal of Differential Equations},
volume = {447},
pages = {113644},
year = {2025},
issn = {0022-0396},
doi = {https://doi.org/10.1016/j.jde.2025.113644},
url = {https://www.sciencedirect.com/science/article/pii/S0022039625006710},
author = {Xueli Bai and Fang Li and Xuefeng Wang}
}

@article{Borgogno2009,
	title = {Mathematical models of vegetation pattern formation in ecohydrology},
	volume = {47},
	issn = {1944-9208},
	url = {http://dx.doi.org/10.1029/2007RG000256},
	doi = {10.1029/2007RG000256},
	number = {1},
	journal = {Reviews of Geophysics},
	author = {Borgogno, F and D'Odorico, P and Laio, F and Ridolfi, L},
	year = {2009},
}

@article{Cantrell1999,
	title = {Diffusion models for population dynamics incorporating individual behavior at boundaries: {Applications} to refuge design},
	volume = {55},
	issn = {00405809},
	doi = {10.1006/tpbi.1998.1397},
	number = {2},
	journal = {Theoretical Population Biology},
	author = {Cantrell, R. S. and Cosner, Chris},
	year = {1999},
	pmid = {10329518},
	pages = {189--207},
}

@book{Cantrell2004,
  title={Spatial ecology via reaction-diffusion equations},
  author={Cantrell, Robert Stephen and Cosner, Chris},
  year={2004},
  publisher={John Wiley \& Sons}
}

@article{CappaneraJaramilloWard2024,
author = {Cappanera, Loic and Jaramillo, Gabriela and Ward, Cory},
year = {2024},
month = {05},
pages = {856-889},
title = {Analysis and Simulation of a Nonlocal Gray–Scott Model},
volume = {84},
journal = {SIAM Journal on Applied Mathematics},
doi = {10.1137/22M1542441}
}

@book{cazenave1998introduction,
  title={An Introduction to Semilinear Evolution Equations},
  author={Cazenave, T. and Haraux, A.},
  isbn={9780198502777},
  lccn={99196876},
  series={Oxford lecture series in mathematics and its applications},
  url={https://books.google.pl/books?id=pcBo9WZBHyMC},
  year={1998},
  publisher={Clarendon Press}
}

@article {ChasseigneChavesRossi2006,
    AUTHOR = {Chasseigne, Emmanuel and Chaves, Manuela and Rossi, Julio D.},
     TITLE = {Asymptotic behavior for nonlocal diffusion equations},
   JOURNAL = {J. Math. Pures Appl. (9)},
  FJOURNAL = {Journal de Math\'{e}matiques Pures et Appliqu\'{e}es.
              Neuvi\`eme S\'{e}rie},
    VOLUME = {86},
      YEAR = {2006},
    NUMBER = {3},
     PAGES = {271--291},
      ISSN = {0021-7824},
   MRCLASS = {35R05 (35A15 35B40 45K05)},
  MRNUMBER = {2257732},
MRREVIEWER = {Rodica\ Luca},
       DOI = {10.1016/j.matpur.2006.04.005},
       URL = {https://doi.org/10.1016/j.matpur.2006.04.005},
}

@article{Clerc2021,
	title = {Localised labyrinthine patterns in ecosystems},
	volume = {11},
	issn = {2045-2322},
	url = {https://www.nature.com/articles/s41598-021-97472-4},
	doi = {10.1038/s41598-021-97472-4},
	language = {en},
	number = {1},
	urldate = {2025-04-14},
	journal = {Scientific Reports},
	author = {Clerc, M. G. and Echeverría-Alar, S. and Tlidi, M.},
	month = sep,
	year = {2021},
	pages = {18331},
}

@article{Colombo2018,
	title = {Nonlinear population dynamics in a bounded habitat},
	volume = {446},
	issn = {00225193},
	url = {https://linkinghub.elsevier.com/retrieve/pii/S0022519318300912},
	doi = {10.1016/j.jtbi.2018.02.030},
	language = {en},
	urldate = {2025-04-14},
	journal = {Journal of Theoretical Biology},
	author = {Colombo, E.H. and Anteneodo, C.},
	month = jun,
	year = {2018},
	pages = {11--18},
}

@article{Dornelas2024,
	title = {Movement bias in asymmetric landscapes and its impact on population distribution and critical habitat size},
	volume = {480},
	copyright = {http://creativecommons.org/licenses/by/4.0/},
	issn = {1364-5021, 1471-2946},
	url = {https://royalsocietypublishing.org/doi/10.1098/rspa.2024.0185},
	doi = {10.1098/rspa.2024.0185},
	language = {en},
	number = {2297},
	urldate = {2024-09-17},
	journal = {Proceedings of the Royal Society A: Mathematical, Physical and Engineering Sciences},
	author = {Dornelas, Vivian and De Castro, Pablo and Calabrese, Justin M. and Fagan, William F. and Martinez-Garcia, Ricardo},
	month = sep,
	year = {2024},
	note = {Publisher: The Royal Society},
}

@article{Eigentler2018,
  title={Analysis of a model for banded vegetation patterns in semi-arid environments with nonlocal dispersal},
  author={Eigentler, Lukas and Sherratt, Jonathan A},
  journal={Journal of mathematical biology},
  volume={77},
  number={3},
  pages={739--763},
  year={2018},
  publisher={Springer}
}

@book{EngelNagel2000,
author = {Engel, Klaus-Jochen and Nagel, Rainer},
year = {2000},
month = {01},
pages = {},
title = {One-Parameter Semigroups for Linear Evolution Equations},
volume = {63},
isbn = {0-387-98463-1},
journal = {Semigroup Forum},
doi = {10.1007/b97696}
}

@book {EvansGariepy2015,
    AUTHOR = {Evans, Lawrence C. and Gariepy, Ronald F.},
     TITLE = {Measure theory and fine properties of functions},
    SERIES = {Textbooks in Mathematics},
   EDITION = {Revised},
 PUBLISHER = {CRC Press, Boca Raton, FL},
      YEAR = {2015},
     PAGES = {xiv+299},
      ISBN = {978-1-4822-4238-6},
   MRCLASS = {28-01},
  MRNUMBER = {3409135},
}

@article{Fagan2009,
	title = {Interspecific variation in critical patch size and gap-crossing ability as determinants of geographic range size distributions},
	volume = {173},
	issn = {00030147},
	doi = {10.1086/596537},
	number = {3},
	journal = {American Naturalist},
	author = {Fagan, William F. and Cantrell, Robert Stephen and Cosner, Chris and Ramakrishnan, Subramanian},
	year = {2009},
	pmid = {19159262},
	pages = {363--375},
}

@book{gilbarg2001elliptic,
  title={Elliptic Partial Differential Equations of Second Order},
  author={Gilbarg, D. and Trudinger, N.S.},
  isbn={9783540411604},
  lccn={00052272},
  series={Classics in Mathematics},
  url={https://books.google.pl/books?id=eoiGTf4cmhwC},
  year={2001},
  publisher={Springer Berlin Heidelberg}
}

@article{JaramilloMeraz2024,
author = {Jaramillo, Gabriela and Meraz, Cristian},
year = {2024},
month = {12},
pages = {},
title = {Existence of Weak Solutions for a Nonlocal Klausmeier Model},
doi = {10.48550/arXiv.2412.14395}
}

@article{Kierstead1953,
	title = {The size of water masses containing plankton blooms},
	volume = {12},
	number = {1},
	journal = {Journal of Marine Research},
	author = {Kierstead, Henry and Slobodkin, L. Basil},
	year = {1953},
}

@article{Klausmeier1999,
	title = {Regular and {Irregular} {Patterns} in {Semiarid} {Vegetation}},
	volume = {284},
	url = {http://www.sciencemag.org/content/284/5421/1826.abstract},
	doi = {10.1126/science.284.5421.1826},
	number = {5421},
	journal = {Science},
	author = {Klausmeier, Christopher A.},
	month = jun,
	year = {1999},
	pages = {1826--1828},
}

@article{LaurencotWalker2023,
title = {A nonlocal Gray-Scott model: Well-posedness and diffusive limit},
journal = {Discrete and Continuous Dynamical Systems - S},
volume = {16},
number = {12},
pages = {3709-3732},
year = {2023},
issn = {1937-1632},
doi = {10.3934/dcdss.2023158},
url = {https://www.aimsciences.org/article/id/64ec6acec82b5f110a3cd267},
author = {Philippe Laurençot and Christoph Walker}
}

@article{Lefever1997,
	title = {On the origin of tiger bush},
	volume = {59},
	url = {http://link.springer.com/article/10.1007/BF02462004},
	number = {2},
	urldate = {2013-06-11},
	journal = {Bulletin of Mathematical Biology},
	author = {Lefever, René and Lejeune, Olivier},
	year = {1997},
	pages = {263--294},
}

@article{Martinez-Garcia2014,
	title = {Minimal mechanisms for vegetation patterns in semiarid regions},
	volume = {372},
	copyright = {All rights reserved},
	issn = {1364503X},
	doi = {10.1098/rsta.2014.0068},
	number = {2027},
	journal = {Philosophical Transactions of the Royal Society A},
	author = {Martinez-Garcia, Ricardo and Calabrese, Justin M. and Hernández-García, Emilio and López, Cristóbal},
	month = oct,
	year = {2014},
	note = {arXiv: 1402.1077v1
Publisher: Royal Society of London
ISBN: 1471-2962},
	pages = {20140068},
}

@article{Martinez-Garcia2022,
	title = {Spatial patterns in ecological systems: from microbial colonies to landscapess},
	volume = {6},
	url = {https://portlandpress.com/emergtoplifesci/article-abstract/6/3/245/231411/Spatial-patterns-in-ecological-systems-from},
	doi = {10.1042/ETLS20210282},
	number = {3},
	journal = {Emerging Topics in Life Sciences},
	author = {Martinez-Garcia, Ricardo and Tarnita, Corina E. and Bonachela, Juan A.},
	year = {2022},
	pages = {245--258},
}

@article{Martinez-Garcia2023,
	title = {Integrating theory and experiments to link local mechanisms and ecosystem-level consequences of vegetation patterns in drylands},
	volume = {166},
	copyright = {All rights reserved},
	url = {https://doi.org/10.48550/arXiv.2101.07049},
	journal = {Chaos, Solitons and Fractals},
	author = {Martinez-Garcia, Ricardo and Cabal, Ciro and Calabrese, Justin M. and Hernández-García, Emilio and Tarnita, Corina E. and López, Cristóbal and Bonachela, Juan A.},
	year = {2023},
	note = {arXiv: 2101.07049},
	pages = {112881},
}

@article{Meron2004,
	title = {Vegetation patterns along a rainfall gradient},
	volume = {19},
	issn = {09600779},
	doi = {10.1016/S0960-0779(03)00049-3},
	number = {2},
	journal = {Chaos, Solitons and Fractals},
	author = {Meron, Ehud and Gilad, Erez and Von Hardenberg, Jost and Shachak, Moshe and Zarmi, Yair},
	year = {2004},
	note = {ISBN: 0960-0779},
	pages = {367--376},
}

@article{Meron2016,
	title = {Pattern formation - {A} missing link in the study of ecosystem response to environmental changes},
	volume = {271},
	issn = {18793134},
	url = {http://dx.doi.org/10.1016/j.mbs.2015.10.015},
	doi = {10.1016/j.mbs.2015.10.015},
	journal = {Mathematical Biosciences},
	author = {Meron, Ehud},
	year = {2016},
	pmid = {26529391},
	note = {Publisher: Elsevier Inc.},
	pages = {1--18},
}

@article{Meron2018,
	title = {From {Patterns} to {Function} in {Living} {Systems}: {Dryland} {Ecosystems} as a {Case} {Study}},
	volume = {9},
	url = {https://doi.org/10.1146/annurev-conmatphys-033117-053959},
	doi = {10.1146/annurev-conmatphys},
	number = {November 2017},
	urldate = {2017-12-12},
	journal = {Annu. Rev. Condens. Matter Phys},
	author = {Meron, Ehud},
	year = {2018},
	pages = {79--103},
}

@book{Pazy,
  author = {Pazy, Amnon},
  title = {Semigroups of Linear Operators and Applications to Partial Differential Equations},
  publisher = {Springer-Verlag},
  year = {1983},
  address = {New York},
  series = {Applied Mathematical Sciences},
  volume = {44}
}

@article{PintoRamos2023,
	title = {Topological defects law for migrating banded vegetation patterns in arid climates},
	volume = {9},
	issn = {23752548},
	doi = {10.1126/sciadv.adf6620},
	number = {31},
	journal = {Science Advances},
	author = {Pinto-Ramos, D. and Clerc, M. G. and Tlidi, M.},
	year = {2023},
	pmid = {37540750},
	pages = {1--12},
}

@article{PintoRamos2025,
  title={How spatial patterns can lead to less resilient ecosystems},
  author={Pinto-Ramos, David and Martinez-Garcia, Ricardo},
  journal={arXiv preprint arXiv:2505.08671},
  year={2025}
}

@article{ReimerBonsallMaini2015,
author = {Reimer, Jody and Bonsall, Michael and Maini, Philip},
year = {2015},
month = {12},
pages = {},
title = {Approximating the Critical Domain Size of Integrodifference Equations},
volume = {78},
journal = {Bulletin of Mathematical Biology},
doi = {10.1007/s11538-015-0129-x}
}

@article{Rietkerk2002,
	title = {Self‐{Organization} of {Vegetation} in {Arid} {Ecosystems}},
	volume = {160},
	issn = {0003-0147},
	url = {http://www.ncbi.nlm.nih.gov/pubmed/18707527},
	doi = {10.1086/342078},
	number = {4},
	urldate = {2017-12-12},
	journal = {The American Naturalist},
	author = {Rietkerk, Max and Boerlijst, Maarten C. and van Langevelde, Frank and HilleRisLambers, Reinier and de Koppel, Johan van and Kumar, Lalit and Prins, Herbert H. T. and de Roos, André M.},
	month = oct,
	year = {2002},
	pmid = {18707527},
	note = {ISBN: 0003-0147},
	pages = {524--530},
}

@article{Rietkerk2008,
	title = {Regular pattern formation in real ecosystems},
	volume = {23},
	url = {http://linkinghub.elsevier.com/retrieve/pii/S0169534708000281},
	number = {3},
	journal = {Trends in ecology \& evolution},
	author = {Rietkerk, Max and van de Koppel, Johan},
	month = mar,
	year = {2008},
	note = {Publisher: Elsevier Science Publishers
ISBN: 0169-5347},
	pages = {169--175},
}

@article{Rietkerk2021,
	title = {Evasion of tipping in complex systems through spatial pattern formation},
	volume = {374},
	issn = {10959203},
	doi = {10.1126/science.abj0359},
	number = {6564},
	journal = {Science},
	author = {Rietkerk, Max and Bastiaansen, Robbin and Banerjee, Swarnendu and van de Koppel, Johan and Baudena, Mara and Doelman, Arjen},
	year = {2021},
	pmid = {34618584},
	pages = {eabj0359},
}

@article{Skellam1951,
	title = {Random {Dispersal} in {Theoretical} {Populations}},
	volume = {38},
	number = {1},
	journal = {Biometrika},
	author = {Skellam, J.G.},
	year = {1951},
	pages = {196--218},
}

@article {Sun2023,
    AUTHOR = {Sun, Jian-Wen},
     TITLE = {Sharp patterns for some semilinear nonlocal dispersal
              equations},
   JOURNAL = {J. Anal. Math.},
  FJOURNAL = {Journal d'Analyse Math\'{e}matique},
    VOLUME = {149},
      YEAR = {2023},
    NUMBER = {2},
     PAGES = {401--419},
      ISSN = {0021-7670,1565-8538},
   MRCLASS = {35K57},
  MRNUMBER = {4594396},
       DOI = {10.1007/s11854-022-0242-3},
       URL = {https://doi.org/10.1007/s11854-022-0242-3},
}

@article{Tadej2025,
title = {Long time behaviour of solutions to non-local and non-linear dispersal problems},
journal = {Journal of Differential Equations},
volume = {416},
pages = {2043-2064},
year = {2025},
issn = {0022-0396},
doi = {https://doi.org/10.1016/j.jde.2024.10.046},
url = {https://www.sciencedirect.com/science/article/pii/S0022039624007095},
author = {Maciej Tadej}
}

@article{Tarnita2024,
	title = {Self-organization in spatial ecology},
	volume = {34},
	issn = {09609822},
	url = {https://linkinghub.elsevier.com/retrieve/pii/S0960982224012430},
	doi = {10.1016/j.cub.2024.09.032},
	language = {en},
	number = {20},
	urldate = {2024-10-22},
	journal = {Current Biology},
	author = {Tarnita, Corina E.},
	month = oct,
	year = {2024},
	pages = {R965--R970},
}

@article{VonHardenberg2001,
	title = {Diversity of {Vegetation} {Patterns} and {Desertification}},
	volume = {87},
	issn = {0031-9007},
	shorttitle = {Phys. {Rev}. {Lett}.},
	url = {http://www.ncbi.nlm.nih.gov/pubmed/11690457},
	doi = {10.1103/PhysRevLett.87.198101},
	number = {19},
	urldate = {2013-11-14},
	journal = {Physical Review Letters},
	author = {von Hardenberg, Jost and Meron, Ehud and Shachak, Moshe and Zarmi, Yair},
	month = oct,
	year = {2001},
	pmid = {11690457},
	note = {Publisher: American Physical Society
ISBN: 0031-9007},
	pages = {198101},
}

@misc{BifurcationKit.jl,
  title = {{BifurcationKit.jl}},
  author = {Veltz, Romain},
  url = {https://hal.archives-ouvertes.fr/hal-02902346},
  institution = {{Inria Sophia-Antipolis}},
  year = {2020},
  month = Jul,
  hal_id = {hal-02902346},
  hal_version = {v1},
  eprint = {hal-02902346},
  eprinttype = {hal},
}

\end{document}